\documentclass[a4paper,oneside,reqno]{amsart}
\usepackage[T1]{fontenc}
\pdfoutput=1
\usepackage[latin1]{inputenc}
\usepackage{graphicx}
\usepackage{amssymb}

\makeatletter
 \theoremstyle{plain}
 \newtheorem{thm}{Theorem}[section]
  \theoremstyle{plain}
  \newtheorem{cor}[thm]{Corollary}
  \theoremstyle{plain}
  \newtheorem{prop}[thm]{Proposition}
  \theoremstyle{remark}
  \newtheorem{rem}[thm]{Remark}
 \theoremstyle{definition}
 \newtheorem*{defn*}{Definition}
  \theoremstyle{plain}
  \newtheorem{lem}[thm]{Lemma}

\usepackage{subfigure}

\usepackage{amscd}

\usepackage{psfrag}

\renewcommand{\emptyset}{\mathchar"001F}

\newcommand{\e}{\mathrm e}

\renewcommand{\hat}{\widehat}
\renewcommand{\phi}{\varphi}
\renewcommand{\tilde}{\widetilde}

\DeclareMathOperator{\diam}{diam}

\newcommand{\R}{{\mathbb R}}
\newcommand{\N}{{\mathbb N}}

\renewcommand{\S}{{\mathbb S}}
\makeatother

\begin{document}

\title[ %
Sets of non-differentiability for conjugacies]{ Sets 
of non-differentiability for conjugacies between expanding 
interval maps}

\author{T. Jordan, M. Kesseböhmer, M. Pollicott, and B.O. Stratmann}

\date{\today{}}

\address{Department of Mathematics, University of Bristol, Bristol, BS8 1TW,
UK}

\email{thomas.jordan@bristol.ac.uk}

\address{Fachbereich 3 - Mathematik und Informatik, Universität Bremen, D--28359
Bremen, Germany}

\email{mhk@math.uni-bremen.de}

\address{Mathematics Institute, University of Warwick, Coventry, CV4 7AL,
UK}

\email{mpollic@maths.warwick.ac.uk}

\address{Mathematics Institute, University of St Andrews, St Andrews, KY16
9SS, Scotland}

\email{bos@maths.st-and.ac.uk}
\thanks{Some of the work for this paper was undertaken during  workshops at the Banach centre in Warsaw and  the Schr\"{o}dinger institute in Vienna in December 2007. The authors would like to thank the institutions for the hospitality shown.}
\begin{abstract}
We study  differentiability of topological conjugacies between 
expanding piecewise $C^{1+\epsilon}$ interval maps. If these conjugacies are not $C^1$,
then they 
have zero derivative almost everywhere. We obtain the result that in this 
case the Hausdorff dimension of
the set of 
points for which the derivative of the conjugacy  does not exist lies strictly between zero and 
one. Using multifractal analysis and thermodynamic formalism, we show that 
this Hausdorff dimension 
is  explicitly  determined by  the Lyapunov spectrum.  Moreover, we 
show that these results give rise to a ``rigidity dichotomy'' for the 
type of conjugacies under consideration.
\end{abstract}
\maketitle

\section{Introduction and statement of results}

In this paper we study aspects of non-differentiability for conjugacy
maps between certain interval maps. The maps under consideration are
called expanding piecewise $C^{1+\epsilon}$ maps. These are expanding
maps of the unit interval $\mathcal{U}$ into itself which have precisely
$d$ increasing full inverse branches and each of these branches is
a $C^{1+\epsilon}$ diffeomorphism on $\mathcal{U}$, for some fixed
$\epsilon>0$ and some fixed integer $d\geq2$ (a map $f:\mathcal{U}\to f\left(\mathcal{U}\right)\subset\R$
is said to be a $C^{1+\epsilon}$ diffeomorphism if there exists an
extension $\tilde{f}$ of $f$ to some open neighbourhood of $\mathcal{U}$
which is a diffeomorphism such that $\tilde{f}'|_{\mathcal{U}}$ is
Hölder continuous with Hölder exponent equal to $\epsilon$). Clearly,
each expanding piecewise $C^{1+\epsilon}$ map is naturally semi-conjugate
to the full shift $\Sigma$ over the alphabet $A:=\left\{ 1,\ldots,d\right\} $.
Moreover, for two maps $S$ and $T$ of this type the following diagram
commutes, where $\sigma$ refers to the usual shift map on $\Sigma$,
and $\pi_{S}$ and $\pi_{T}$ denote the associated coding maps. \[
\begin{CD}\mathcal{U}@<\pi_{T}<<\Sigma@>\pi_{S}>>\mathcal{U}\\
@VTVV@VV\sigma V@VVSV\\
\mathcal{U}@<\pi_{T}<<\Sigma@>\pi_{S}>>\mathcal{U}\end{CD}\]
 The conjugacy map $\Theta:\mathcal{U}\to\mathcal{U}$ between the
two systems $(\mathcal{U},S)$ and $(\mathcal{U},T)$ is then given
by $T\circ\Theta=\Theta\circ S$ (see Fig. \ref{fig:Conjugacy-maps1}
and \ref{fig:Conjugacy-maps} for some examples). The first main result
of the paper will be to employ the thermodynamic formalism in order
to give a detailed fractal analysis of the following three sets:
 \begin{eqnarray*}
\mathcal{D}_{\sim}\hspace{-2.7mm} & = & \hspace{-2.7mm}\mathcal{D}_{\sim}(S,T):=\{\xi\in\mathcal{U}:\Theta'(\xi)\mbox{ does not exists in the generalised sense}\};\\
\mathcal{D}_{\infty}\hspace{-2.7mm} & = & \hspace{-2.7mm}\mathcal{D}_{\infty}(S,T):=\{\xi\in\mathcal{U}:\Theta'(\xi)=\infty\};\\
\mathcal{D}_{0}\hspace{-1.7mm} & = & \hspace{-2.7mm}\mathcal{D}_{0}(S,T):=\{\xi\in\mathcal{U}:\Theta'(\xi)=0\},
\end{eqnarray*}
 where
\emph{$\Theta'(\xi)$ exists in the generalised sense} means that
$\Theta'(\xi)$ either exists or else is equal to infinity
(at the boundary points we interpret these quantities in terms of
limits from the left or right, as appropriate).
Note that we can trivially write 
 $\mathcal{U}=\mathcal{D}_{\sim}
\cup\mathcal{D}_{\infty}\cup\mathcal{D}_{0}\cup\mathcal{D}$
where $\mathcal{D}:=\{\xi\in\mathcal{U}:\Theta'(\xi)\in (0,\infty)\}$.
 However, as we  will see,  either $\mathcal{D}=\mathcal{U}$ or  $\mathcal{D}=\emptyset$.

The second main result of the paper will be to give a necessary and
sufficient condition for when two expanding piecewise $C^{1+\epsilon}$
systems $(\mathcal{U},S)$ and $(\mathcal{U},T)$ are rigid in a certain
sense.

To state our main results in greater detail, let us define the Hölder
continuous potentials $\phi,\psi:\Sigma\to\R_{<0}$ for $x=(x_{1}x_{2}...)\in\Sigma$
by \[
\phi\left(x\right):=\log\left(S_{x_{1}}^{-1}\right)'\left(\pi_{S}\left(\sigma(x)\right)\right)\,\mbox{ and }\,\psi\left(x\right):=\log\left(T_{x_{1}}^{-1}\right)'\left(\pi_{T}\left(\sigma(x)\right)\right),\]
 where $S_{a}^{-1}$ and $T_{a}^{-1}$ denote the inverse branches
of $S$ and $T$ associated with $a\in A$. Let $\beta:\R\to\R$ be
defined implicitly by the pressure equation\[
P\left(s\phi+\beta\left(s\right)\psi\right)=0,\,\mbox{ for }\, s\in\R.\]
 Note that $\beta$ is well defined, since $\psi<0$. We let $\mu_{s}$
denote the equilibrium measure associated with the potential function
$s\phi+\beta\left(s\right)\psi$. Since \[
\beta'\left(s\right):=\frac{-\int\phi\, d\mu_{s}}{\int\psi\, d\mu_{s}}<0,\]
 we have that $\beta$ is strictly decreasing. Moreover, $\beta\left(1\right)=0$
and $\beta\left(0\right)=1$. If $\phi$ and $\psi$  are {\em cohomologically
independent}, that is, if there are no nontrivial choices of $b,c\in\R$
and $u\in C(\Sigma)$ such that $b\phi+c\psi=u\circ\sigma-u$ (in this 
situation, we will also say that $S$ and $T$ are cohomologically
independent), then
we have that $\beta$ is strictly convex (see e.g. \cite{Pesin}). 
Hence, if $S$ and $T$  are cohomologically
independent, then we have  by the mean value theorem for derivatives 
that  there exists a unique
number $s_{0}\in(0,1)$ such that $\beta'\left(s_{0}\right)=-1$.
For ease of exposition, we define the function $\tilde{\beta}:\R\to\R$
by $\tilde{\beta}\left(s\right):=\beta\left(s\right)+s$. Note that
$\tilde{\beta}$ is convex and has a unique minimum at $s_{0}$. Moreover,
we have $\tilde{\beta}(0)=\tilde{\beta}\left(1\right)=1$ and $\tilde{\beta}\left(s_{0}\right)=\hat{\beta}\left(1\right)$,
where $\hat{\beta}$ denotes the (concave) Legendre transform of $\beta$, given
by $\hat{\beta}(s):=\inf_{t\in\R}(\beta(t)+st)$, for $s\in\R$. Finally,
the level sets $\mathcal{L}\left(s\right)$ are defined by \[
\mathcal{L}\left(s\right):=\left\{ \xi\in\mathcal{U}:\lim_{n\to\infty}\frac{S_{n}\phi\left(\xi\right)}{S_{n}\psi\left(\xi\right)}=s\right\} .\]
 By standard thermodynamic formalism (see e.g. \cite{Pesin}), we
then have for $s$ in the closure $\overline{(-\beta'(\R))}$ of the
domain of $-\beta'$ that \[
\dim_{H}\left(\mathcal{L}\left(s\right)\right)=\hat{\beta}\left(s\right)/s=\frac{1}{s}\inf_{t\in\R}\left(st+\beta\left(t\right)\right)=\inf_{t\in\R}\left(t+\beta\left(t\right)/s\right),\]
 whereas for $s\notin\overline{(-\beta'(\R))}$ we have $\mathcal{L}(s)=\emptyset$.\\
 The first main results of this paper are now stated in the following
theorem.

\begin{thm}
\label{main} Let $S$ and $T$ be two cohomologically independent
expanding piecewise $C^{1+\epsilon}$ maps of the unit interval into
itself. We then have that \[
0<\dim_{H}\left(\mathcal{D}_{\sim}\right)=\dim_{H}\left(\mathcal{D}_{\infty}\right)=\dim_{H}\left(\mathcal{L}\left(1\right)\right)=\tilde{\beta}\left(s_{0}\right)<1.\]

\end{thm}
\vspace{2mm}
 Our second main result is that for the type of interval maps which
we consider in this paper, one has the following rigidity theorem.
Here, $\lambda$ denotes the Lebesgue measure on $\mathcal{U}$.

\begin{thm}
\emph{\label{thm:Rigidity}} Let $S$ and $T$ be two expanding piecewise
$C^{1+\epsilon}$ maps of the unit
interval into itself. We then have that \[
\Theta\,\mbox{is a }C^{1+\epsilon}\mbox{ diffeomorphism if and only if }\,\dim_{H}\left(\mathcal{D}_{\sim}\right)=0.\]
 More precisely, we have that the following ``rigidity dichotomy'' holds. 
\begin{enumerate}
\item If $S$ and $T$ are cohomologically dependent, then $\Theta$ is
a $C^{1+\epsilon}$ diffeomorphism and hence absolutely continuous.
Equivalently, we have that \[
\mathcal{D}_{0}=\mathcal{D}_{\infty}=\mathcal{D}_{\sim}=\emptyset,\,\mbox{ and hence }\,\mathcal{U}=\{\xi\in\mathcal{U}:0<\Theta'(\xi)<\infty\}.\]

\item If $S$ and $T$ are cohomologically independent, then the conjugacy
$\Theta$ is singular, that is, $\lambda\left(\mathcal{D}_{0}\right)=1$.
Moreover, $\Theta$ is Hölder continuous with Hölder exponent equal
to $\left(\sup_{x\in\R}-\beta'(x)\right)^{-1}$, and we have that
\[
0<\dim_{H}\left(\mathcal{D}_{\infty}\right)=\dim_{H}\left(\mathcal{D}_{\sim}\right)<1.\]

\end{enumerate}
\end{thm}
The latter theorem is closely related to classical work by Shub and
Sullivan \cite{ShubSullivan} addressing the smoothness of conjugacies
between expanding maps of the unit circle $\S^{1}$ (see also e.g.
\cite{Bowen} \cite{J}  \cite{Mostow} \cite{Sullivan}). In \cite{ShubSullivan}
is was shown for $k\geq2$ that if the conjugacy between two $C^{k}$expanding
maps is absolutely continuous then it is necessarily $C^{k-1}$. Let
us also mention a result by Cui \cite{Guizhen} which states that
the conjugacy map between two expanding $C^{1+\epsilon}$ circle endomorphisms
is itself $C^{1+\epsilon}$, if it has finite, nonzero derivative
at some point in $\S^{1}$. So, to deduce Theorem \ref{thm:Rigidity}
from Theorem \ref{main}, we need to adapt this result to the setting
of interval maps. In the case of circle maps we can use our result
on interval maps and the result of Cui to obtain a result
for endomorphisms of $\S^{1}$. For this note that Theorem \ref{main}
can be adapted such that it is applicable to the situation in which
the two dynamical systems are orientation preserving expanding $C^{1+\epsilon}$
circle maps. This gives rise to the following result. 

\begin{cor}
{\label{thm:Rigidity2}For the conjugacy map $\Phi$ between a given
pair $(\S^{1},U)$ and $(\S^{1},V)$ of expanding $C^{1+\epsilon}$
endomorphisms of $\S^{1}$, the following statements are equivalent. 
\begin{enumerate}
\item $\Phi$ is a $C^{1+\epsilon}$ circle map; 
\item $\dim_{H}\left(\{\xi\in\S^{1}:\Phi'(\xi)\mbox{ does not exists in the generalised sense}\}\right)=0$; 
\item $\dim_{H}\left(\{\xi\in\S^{1}:0<\Phi'(\xi)<\infty\}\right)=1$; 
\item $\Phi$ is absolutely continuous; 
\item $\Phi$ is bi-Lipschitz. 
\end{enumerate}
} 
\end{cor}
A natural question to ask is how the Hausdorff dimensions of the sets
$\mathcal{D}_{\infty}(S,T)$ and $\mathcal{D}_{\sim}(S,T)$ vary as
$S$ and $T$ change. The next two results address this question.

\begin{prop}
\label{prop1} For a $C^{k}$ family of expanding maps we have that
the Hausdorff dimension of the non-differentiability set has a $C^{k-2}$
dependence. 
\begin{prop}
\label{prop2} There exists a pair of $C^{2}$ circle-endomorphisms
for which the set of non-differentiable points for the associated
conjugacy map has arbitrary small Hausdorff dimension. 
\end{prop}
\end{prop}
The paper is organised as follows. In Section \ref{sec:Proof-of-Theorem1}
and Section \ref{sec:Proof-of-Theorem2} we give the proofs of Theorem
\ref{main} and Theorem \ref{thm:Rigidity}. Section \ref{sec:Examples-and-proof}
discusses two basic examples, and one of these is then used in Section
\ref{sec:C^{k-2}-dependence} for the proof of Proposition \ref{prop2}.
Moreover, in Section \ref{sec:C^{k-2}-dependence} we study the dependence
of the dimension of non-differentiable points and give the proof of
Proposition \ref{prop1}. 

\begin{rem}
\label{r:1}~

(1) Note that \[
\mathcal{D}_{\sim}(S,T)\cup\mathcal{D}_{\infty}(S,T)
=\{\xi\in\mathcal{U}:\Theta\text{ is not differentiable at }\xi\},\]
 and hence, Theorem \ref{main} in particular implies that if $S$ and $T$ 
 are cohomologically independent, then the Hausdorff
dimension of the set of points for which $\Theta$ is not differentiable
is equal to $\tilde{\beta}\left(s_{0}\right)$. 

(2) There is a variational formula for the Hausdorff dimension of
the set $\mathcal{D}_{\sim}$. Namely, as we will see in Section 2.3, 
we have that \[
\dim_{H}(\mathcal{D}_{\sim})=\sup\left\{ \frac{h(\mu)}{\int\phi d\mu}:\frac{\int\phi d\mu}{\int
\psi d\mu}=1\right\} ,\]
where the supremum ranges over all $\sigma$-invariant probability measures on $\Sigma$.  
From this formula it is clear that if we swap the roles of $\phi$
and $\psi$, then this has no effect on the dimension of the set of
non-differentiability. In other words, if instead of $\Theta$ we
take the dual conjugacy $\hat{\Theta}$, given by $S\circ\hat{\Theta}=\hat{\Theta}\circ T$,
then the Hausdorff dimension of the set of points at which $\hat{\Theta}'$
does not exist in the generalised sense coincides with $\dim_{H}\left(\mathcal{D}_{\sim}\right)$,
i.e. $\dim_{H}\left(\mathcal{D}_{\sim}(S,T)\right)=\dim_{H}\left(\mathcal{D}_{\sim}(T,S)\right)$.

(3) The conjugacy map $\Theta$ can also be viewed as the distribution
function of the measure $m_\Theta:=\lambda\circ\Theta$.
This follows, since for $\xi\in  \mathcal{U}$ we have $$m_\Theta \left(\left[0,\xi\right)\right)=
\lambda\left(\left[0 ,\Theta\left(\xi\right)\right)\right)=\Theta\left(\xi\right).$$
Hence, the investigations in this paper  can also be seen as a study of singular distribution functions
which are 
supported on whole unit interval $\mathcal{U}$. Note that there are 
strong parallels to the results in \cite{KessStrat2}, where we used 
some of the outcomes of \cite{KessStrat4} to give a fractal analysis of non 
differentibility  for Minkowski's question mark function. 

(4) Finally, let us mention that  the statements in Theorem \ref{main}
and Theorem \ref{thm:Rigidity} can be generalised so that the derivative of $\Theta$ gets replaced by   the $s$-Hölder derivative $\Delta_s \Theta$ of $\Theta$,  given for $ s \in-\beta'\left(\R\right)$ by
 \[
\left(\Delta_s \Theta \right) \left(\xi\right):=\lim_{\eta\to\xi}\frac{\left|\Theta\left(\eta\right)-\Theta\left(\xi\right)\right|}{\left|\eta-\xi\right|^{s}}.\]
For this more general derivative the relevant sets are \begin{eqnarray*}
\mathcal{D}_{\sim}^{(s)}= \mathcal{D}_{\sim}^{(s)}  (S,T)\hspace{-2.7mm} & := &\hspace{-2.7mm} \{\xi\in\mathcal{U}:\left(\Delta_s \Theta\right) (\xi)\mbox{ does not exists in the generalised sense}\},\\
\mathcal{D}_{\infty}^{(s)}= \mathcal{D}_{\infty}^{(s)}(S,T)\hspace{-2.7mm} & := &\hspace{-2.7mm} \{\xi\in\mathcal{U}:\left(\Delta_s \Theta\right) (\xi)=\infty\}.\end{eqnarray*}
Straightforward adaptations of the proofs in this paper then show that  \[
\dim_{H}\left(\mathcal{D}_{\sim}^{(s)}\right)=\dim_{H}\left(\mathcal{D}_{\infty}^{(s)}\right)=\dim_{H}\left(\mathcal{L}\left(s \right)\right).\]
This shows that on $-\beta'\left(\R\right)$ the  Lyapunov spectrum $s \mapsto \hat{\beta}\left(s\right)/s$  coincides with the 
``\emph{spectrum of non $s$-Hölder differentiability of $\Theta$}''. Note 
that for certain Cantor-like sets similar results   
were obtained in  \cite{KessStrat1},  where we derived generalisations of results of \cite{Darst:93}, \cite{Falconer:04}
and others.
\end{rem}

\section{\label{sec:Proof-of-Theorem1} Proof of Theorem \ref{main}}

\subsection{The geometry of the derivative of $\Theta$}

Let us first introduce some notations which will be used throughout. 

\begin{defn*}
\noindent Let us say that $x=(x_{1}x_{2}\ldots)\in\Sigma$ has an
$i$\emph{-block of length $k$ at the $n$-th level}, for $n,k\in\N$
and $i\in\{1,d\}$, if $x_{n+k+1}\in A\setminus\{i\}$ and $x_{n+m}=i$,
for all $m\in\{1,\ldots,k\}$. Moreover, we will say that $x=(x_{1}x_{2}\ldots)\in\Sigma$
has a \emph{strict} $i$\emph{-block of length $k$ at the $n$-th
level}, if we additionally have that $x_{n}\in A\setminus\{i\}$. 
\end{defn*}
For ease of exposition, we define the function $\chi:\Sigma\rightarrow\R$
by $\chi:=\psi-\phi$. Also, let $D_{\Theta}(\xi,\eta)$ denote the
differential quotient for $\Theta$ at $\xi$ and $\eta$, that is
\[
D_{\Theta}(\xi,\eta):=\frac{\Theta(\xi)-\Theta(\eta)}{\xi-\eta}.\]
 Moreover, we use the notation $\underline{a}_{k}$ to denote the
word of length $k\in\N$ containing exclusively the letter $a\in A$,
and we let $\underline{a}$ denote the infinite word containing exclusively
the letter $a\in A$. Also, $[x_{1}\ldots x_{n}]$ denotes the cylinder
set associated with the finite word $(x_{1},\ldots,x_{n})\in A^{n}$,
that is, \[
[x_{1}\ldots x_{n}]:=\{(y_{1}y_{2}\ldots)\in\Sigma:y_{i}=x_{i},\mbox{ for all }\, i=1,\ldots,n\}.\]
Throughout,   ` $\asymp$'   means that the ratio of the left hand side to the
right hand side is uniformly bounded away from zero and infinity.  Likewise, 
we use  $\ll$ to denote that the expression on the left hand side is uniformly bounded
by the expression on the right hand side multiplied by some fixed positive constant. 

 Let us begin our discussion of the geometry of the derivative of
$\Theta$ with the following crucial geometric observation. 

\begin{prop}
\label{geometric} Let $x=(x_{1}x_{2}\ldots),y=(y_{1}y_{2}\ldots)\in\Sigma$
satisfy $y\in[x_{1}\ldots x_{n-1}]$ as well as $x_{n}=a$ and $y_{n}=b$
for some $n\in\N$ and $a,b\in A$ with $|a-b|=1$ (note that for
$n=1$ we adopt the convention that $x_{1}=a$ and $y_{1}=b$). Moreover,
assume that for some $k,l\in\N$ we have that $x$ has an $i$-block
of length $k$ at the $n$-th level, and $y$ has a $j$-block of
length $l$ at the $n$-th level. Here, $i,j\in\{1,d\}$ are chosen
such that if $a<b$ then $i=d$ and $j=1$, whereas if $a>b$ then
$i=1$ and $j=d$. In this situation we have for $\xi:=\pi_{S}(x)$
and $\eta:=\pi_{S}(y)$, \[
D_{\Theta}(\xi,\eta)\asymp\e^{S_{n}\chi(x)}\,\frac{\e^{k\psi((\underline{i}))}+\e^{l\psi((\underline{j}))}}{\e^{k\phi((\underline{i}))}+\e^{l\phi((\underline{j}))}}.\]
\end{prop}
\begin{proof}
We only consider the case $a=b+1>b$. The case $a<b$ is completely
analogous and is left to the reader. In this situation we then have
for some $p\in A\setminus\{1\}$ and $q\in A\setminus\{d\}$ that
$x$ and $y$ are of the form \[
x=(x_{1}\ldots x_{n-1}a\underline{1}_{k}p\ldots)\,\hbox{\, and\,}\, y=(x_{1}\ldots x_{n-1}b\underline{d}_{l}q\ldots).\]
 Then consider the following cylinder sets \[
I_{1}:=\pi_{S}([x_{1}\ldots x_{n-1}b\underline{d}_{l+1}])\,\hbox{\, and\,}\, I_{2}:=\pi_{S}([x_{1}\ldots x_{n-1}a\underline{1}_{k+1}]),\]
 and \[
J_{1}:=\pi_{S}([x_{1}\ldots x_{n-1}b\underline{d}_{l}])\,\hbox{\, and\,}\, J_{2}:=\pi_{S}([x_{1}\ldots x_{n-1}a\underline{1}_{k}]).\]
 One immediately verifies that for the interval $[\eta,\xi]$ we have
\[
I_{1}\cup I_{2}\subset[\eta,\xi]\subset J_{1}\cup J_{2}.\]
 Moreover, with $\eta':=\pi_{T}(\left(x_{1}\ldots x_{n-1}b\underline{d}\right))=\pi_{T}(\left(x_{1}\ldots x_{n-1}a\underline{1}\right))$
we have, using the bounded distortion property, \[
|\Theta(\xi)-\Theta(\eta)|=|\Theta(\eta)-\Theta(\eta')|+|\Theta(\eta')-\Theta(\eta)|\asymp\e^{S_{n}\psi(x)}\,\left(\e^{k\psi((\underline{d}))}+\e^{l\psi((\underline{1}))}\right).\]
 Similarly, one obtains \[
|\xi-\eta|\asymp \diam(I_{1})+ \diam (I_{2}) \asymp \diam( J_{1})+\diam( J_{2}) \asymp\e^{S_{n}\phi(x)}\,\left(\e^{k\phi((\underline{d}))}+\e^{l\phi((\underline{1}))}\right).\]

\end{proof}
Note that Proposition \ref{geometric} does in particular contain
\emph{all} cases in which $D_{\Theta}(\pi_{S}(x),\pi_{S}(y))$ can
significantly deviate from $\exp(S_{n}\chi(x))$, for given $x,y\in\Sigma$.
This is clarified by the following lemma, which  addresses the cases not 
covered by Proposition \ref{geometric}.  

\begin{lem}
\label{geom2} Let $x=(x_{1}x_{2}\ldots),y=(y_{1}y_{2}\ldots)\in\Sigma$
be given such that $y\in[x_{1}\ldots x_{n-1}]\setminus[x_{1}\ldots x_{n}]$
and such that either $|x_n-y_n|>1$, or if $|x_n-y_n|=1$ then 
$\pi_S([x_1 ... x_{n+1}]) \cap \pi_S([y_1 ... y_{n+1}]) = \emptyset$.  For
$\xi:=\pi_{S}(x)$ and $\eta:=\pi_{S}(y)$, we then have \[
D_{\Theta}(\xi,\eta)\asymp\e^{S_{n}\chi(x)}.\]

\end{lem}
\begin{proof}
Let $x$ and $y$ be given as stated in the lemma. We then have that either $\pi_S([x_1 ... x_n]) \cap \pi_S([y_1 ... y_n])= \emptyset$, and hence there exists an interval separating these to sets, or if   $\pi_S([x_1 ... x_n]) \cap \pi_S([y_1 ... y_n]) \neq  \emptyset$ then  $\pi_S([x_1 ... x_{n+1}]) \cap \pi_S([y_1 ... y_{n+1}])= \emptyset$. Clearly,  in both cases  there exists
$a,b\in A$ such that the interval $I_{ab}:=\pi_{S}([x_{1}\ldots x_{n-1}ab])$
separates the two intervals $\pi_{S}([x_{1}\ldots x_{n+1}])$ and
$\pi_{S}([y_{1}\ldots y_{n+1}])$. Using this, we then obtain
\[
\e^{S_{n}\psi(x)}\ll \diam \left(\Theta(I_{ab})\right) \ll  |\Theta(\xi)-\Theta(\eta)|\ll
\diam \left(\Theta(\pi_{S}([x_{1}\ldots x_{n-1}])) \right) \ll\e^{S_{n}\psi(x)},\]
 and \[
\e^{S_{n}\phi(x)}\ll \diam \left(I_{ab} \right) \ll |\xi-\eta|\ll \diam \left(\pi_{s}([x_{1}\ldots x_{n-1}])\right) \ll \e^{S_{n}\phi(x)}.\]
\end{proof}
\begin{lem}
\label{geom1} If $x=(x_{1}x_{2}\ldots)\in\Sigma$ has an $i$-block
of length $k$ at the $n$-th level, for some $n,k\in\N$ and $i\in\{1,d\}$,
then we have for each $\eta\in\pi_{s}([x_{1}\ldots x_{n}]\setminus[x_{1}\ldots x_{n+1}])$,
with $\delta:=-\min\{\psi((\underline{1})),\psi((\underline{d}))\}>0$
and $\xi:=\pi_{S}(x)$, \[
D_{\Theta}(\xi,\eta)\gg\e^{S_{n}\chi(x)}\e^{-\delta k}.\]
 
\end{lem}
\begin{proof}
Let $\xi$ and $\eta$ be given as stated in the lemma. Trivially,
we have $|\xi-\eta|\ll\exp(S_{n}\phi(x))$. As in the proof of the
previous lemma, one immediately verifies that \[
|\Theta(\xi)-\Theta(\eta)|\gg\e^{S_{n+k}\psi(x)}\gg\e^{S_{n}\psi(x)}\e^{-\delta k}.\]
 By combining these observations, the result follows. 
\end{proof}
\begin{lem}
\label{der}For $x\in\Sigma$ such that $\xi:=\pi_{S}(x)$ the following
hold. 
\begin{enumerate}
\item If $\limsup_{n\rightarrow\infty}\e^{S_{n}\chi(x)}=\infty$, then $\limsup_{\eta\to\xi}D_{\Theta}\left(\xi,\eta\right)=\infty$. 
\item If $\liminf_{n\rightarrow\infty}\e^{S_{n}\chi(x)}=0$, then $\liminf_{\eta\to\xi}D_{\Theta}\left(\xi,\eta\right)=0$. 
\end{enumerate}
\end{lem}
\begin{proof}
Let $\xi$ and $x=(x_{1}x_{2}...)$ be given as stated in the lemma
and assume without loss of generality that $\xi\notin\pi_{S}\left(\left\{ \left(x_{1}x_{2}\ldots\right)\in\Sigma:\:\exists\, n\in\N\,\exists i\in\{1,d\}\,\forall k\geq n:\; x_{k}=i\right\} \right)$.
For $n\in\N$, the left and right boundary points of $\pi_{s}([x_{1}\ldots x_{n}])$
are given by $\xi_{n}:=\pi_{S}((x_{1}\ldots x_{n}\underline{1}))$
and $\eta_{n}:=\pi_{S}((x_{1}\ldots x_{n}\underline{d}))$. By assumption
we have $\xi\notin\left\{ \xi_{n},\eta_{n}:n\in\N\right\} $. It then
follows that \[
\min\left\{ D_{\Theta}\left(\xi,\eta_{n}\right),D_{\Theta}\left(\xi,\xi_{n}\right)\right\} \leq D_{\Theta}\left(\xi_{n},\eta_{n}\right)\leq\max\left\{ D_{\Theta}\left(\xi,\eta_{n}\right),D_{\Theta}\left(\xi,\xi_{n}\right)\right\} .\]
 Since $D_{\Theta}\left(\xi_{n},\eta_{n}\right)\asymp\e^{S_{n}\chi\left(\xi\right)}$,
the lemma follows. 
\end{proof}
We have the following immediate corollary.

\begin{cor}
\label{deri} Let $x\in\Sigma$ be given such that \[
\liminf_{n\rightarrow\infty}\,\e^{S_{n}\chi(x)}=0\,\hbox{and}\,\limsup_{n\rightarrow\infty}\,\e^{S_{n}\chi(x)}=\infty.\]
 We then have that $\pi_{S}(x)\in\mathcal{D}_{\sim}$. 
\end{cor}
For the remainder of this section we restrict the discussion to the
following two cases. As we will see in Lemma \ref{lem:Alternative},
these are in fact the only relevant cases for the purposes in this
paper. \begin{equation}
\textbf{Case 1:}\,\,\,\frac{\psi\left(\underline{1}\right)}{\phi\left(\underline{1}\right)}<\min\left\{ \frac{\psi\left(\underline{d}\right)}{\phi\left(\underline{d}\right)},1\right\} ;\,\,\,\,\textbf{Case 2:}\,\,\,\frac{\psi\left(\underline{d}\right)}{\phi\left(\underline{d}\right)}<\min\left\{ \frac{\psi\left(\underline{1}\right)}{\phi\left(\underline{1}\right)},1\right\} .\label{eq:case1_2}\end{equation}
 In fact, without loss of generality we will always assume that we
are in the situation of Case 1. The discussion of Case 2 is completely
analogous (essentially, one has to interchange the roles of $1$ and
$d$ as well as of $l$ and $k$), and will be left to the reader.
Note that Case 1 and 2 include the cases \[
\e^{\chi((\underline{1}))}>1>\e^{\chi((\underline{d}))}\mbox{ and }\e^{\chi((\underline{1}))}<1<\e^{\chi((\underline{d}))},\]
 which are for instance fulfilled in the Salem-examples briefly discussed
in Section \ref{sec:Examples-and-proof}. On the basis of this assumption,
we now make the following crucial observation.

\begin{lem}
\label{lem:keyInequality} Assume that we are in Case 1 of (\ref{eq:case1_2}).
For all $l\in\N$ we then have \[
\frac{\e^{l\psi((\underline{1}))}+\e^{k\psi((\underline{d}))}}{\e^{l\phi((\underline{1}))}+\e^{k\phi((\underline{d}))}}\ll\e^{\alpha k},\]
 where $\alpha:=\chi((\underline{1}))\phi((\underline{d}))/\phi((\underline{1}))>0$.
Moreover, if $l=\left\lfloor k\,\alpha/\chi((\underline{1}))\right\rfloor $
then \[
\frac{\e^{l\psi((\underline{1}))}+\e^{k\psi((\underline{d}))}}{\e^{l\phi((\underline{1}))}+\e^{k\phi((\underline{d}))}}\gg\e^{\alpha k}.\]
 Here, $\left\lfloor r\right\rfloor $ denotes the smallest integer
greater than or equal to $r\in\R$. 
\end{lem}
\begin{proof}
First note that with $\alpha':=\phi((\underline{d}))/\phi((\underline{1}))$
the conditions in Case 1 immediately imply \[
\e^{\phi((\underline{1}))}<\e^{\psi((\underline{1}))}\;\mbox{ and }\e^{\psi((\underline{d}))}<\e^{\alpha'\psi((\underline{1}))}.\]
 In particular, this implies that $\chi((\underline{1}))>0$. We then
have for all $l\geq\alpha'k$ that \[
\frac{\e^{l\psi((\underline{1}))}+\e^{k\psi((\underline{d}))}}{\e^{l\phi((\underline{1}))}+\e^{k\phi((\underline{d}))}}\leq\frac{\e^{l\psi((\underline{1}))}+\e^{k\alpha'\psi((\underline{1}))}}{\e^{k\phi((\underline{d}))}}\leq2\frac{\e^{k\alpha'\psi((\underline{1}))}}{\e^{k\phi((\underline{d}))}}=2\frac{\e^{k\alpha'\psi((\underline{1}))}}{\e^{k\alpha'\phi((\underline{1}))}}=2\e^{\alpha k}.\]
 If $l\leq\alpha'k$, then we obtain\[
\frac{\e^{l\psi((\underline{1}))}+\e^{k\psi((\underline{d}))}}{\e^{l\phi((\underline{1}))}+\e^{k\phi((\underline{d}))}}\leq\frac{\e^{l\psi((\underline{1}))}+\e^{k\alpha'\psi((\underline{1}))}}{\e^{l\phi((\underline{1}))}}\leq2\left(\frac{\e^{\psi((\underline{1}))}}{\e^{\phi((\underline{1}))}}\right)^{l}\leq2\left(\frac{\e^{\psi((\underline{1}))}}{\e^{\phi((\underline{1}))}}\right)^{\alpha'k}=2\e^{\alpha k}.\]
 Finally, if $l=\left\lfloor \alpha'k\right\rfloor $ then we have
\[
\frac{\e^{l\psi((\underline{1}))}+\e^{k\psi((\underline{d}))}}{\e^{l\phi((\underline{1}))}+\e^{k\phi((\underline{d}))}}\asymp\frac{\e^{l\psi((\underline{1}))}+\e^{k\psi((\underline{d}))}}{2\e^{k\phi((\underline{d}))}}\gg\frac{\e^{k\alpha'\psi((\underline{1}))}}{2\e^{k\alpha'\phi((\underline{1}))}}\asymp\e^{\alpha k}.\]

\end{proof}
For the following proposition we define the two sets \[
\mathcal{D}_{\sharp}:=\pi_{S}\left(\left\{ x\in\Sigma:\lim_{n\to\infty}S_{n}\chi(x)=-\infty\right\} \right)\]
 and \[
\mathcal{D}_{\sharp}^{*}=\mathcal{D}_{\sharp}\setminus\pi_{S}\left(\left\{ \left(x_{1}x_{2}\ldots\right)\in\Sigma:\:\exists\, n\in\N\,\exists i\in\{1,d\}\,\forall k\geq n:\; x_{k}=i\right\} \right).\]

\begin{prop}
\label{Prop3} Assume that we are in Case 1 of (\ref{eq:case1_2}).
Let $x=(x_{1}x_{2}\ldots)\in\Sigma$ be given such that $\xi:=\pi_{S}(x)\in\mathcal{D}_{\sharp}^{*}$.
We then have that $\xi\in\mathcal{D}_{\sim}$ if and only if there
exist strictly increasing sequences $(n_{m})_{m\in\N}$ and $(k_{m})_{m\in\N}$
of positive integers such that $x$ has a strict $d$-block of length
$k_{m}$ at the $n_{m}$-th level for each $m\in\N$, and \[
\e^{S_{n_{m}}\chi(x)+k_{m}\alpha}\gg1,\,\mbox{for all }\, m\in\N.\]

\end{prop}
\begin{proof}
Let $x=(x_{1}x_{2}\ldots)\in\Sigma$ be given such that $\xi:=\pi_{S}(x)\in\mathcal{D}_{\sharp}^{*}$.
We then have by Lemma \ref{der} that there exists a sequence $\left(\eta_{n}\right)_{n\in\N}$
such that \[
\lim_{n\to\infty}D_{\Theta}(\xi,\eta_{n})=0.\]
 Now, for the `if-part' assume that $\xi$ has strict $d$-blocks
as specified in the proposition. For each $m\in\N$, we then choose
$\eta_{m}'$ to be some element of the interval $\pi_{S}([x_{1}\ldots x_{n_{m}}(x_{n_{m}+1}+1)\underline{1}_{l_{m}}a])$,
where $a\in A\setminus\{1\}$ and $l_{m}:=k_{m}\alpha/\chi((\underline{1}))$.
Combining Proposition \ref{geometric}, the second part of Lemma \ref{lem:keyInequality}
and the fact that $\exp(S_{n_{m}}\chi(x)+k_{m}\alpha)\gg1$, we then
obtain \[
D_{\Theta}(\xi,\eta_{m}')\asymp\e^{S_{n_{m}}\chi(x)}\,\frac{\e^{k_{m}\psi((\underline{d}))}+\e^{l_{m}\psi((\underline{1}))}}{\e^{k_{m}\phi((\underline{d}))}+\e^{l_{m}\phi((\underline{1}))}}\asymp\e^{S_{n_{m}}\chi(x)+k_{m}\alpha}\gg1,\,\mbox{for all }\, m\in\N.\]
 Combining this with the observation at the beginning of the proof,
it follows that $\xi\in\mathcal{D}_{\sim}$.\\
 For the `only-if-part', let $x=(x_{1}x_{2}...)\in\Sigma$ be given
such that $\xi:=\pi_{S}(x)\in\mathcal{D}_{\sim}\cap D_{\sharp}^{*}.$
Then there exists a sequence $\left(\eta_{m}\right)_{m\in\N}$ in
$\mathcal{U}$ and a strictly increasing sequence $(n_{m})_{m\in\N}$
in $\N$ such that for all $m\in\N$ we have $\eta_{m}\in\pi_{S}([x_{1}...x_{n_{m}}])$
and \[
\liminf_{m\rightarrow\infty}D_{\Theta}(\xi,\eta_{m})>0.\]
 Using Proposition \ref{geometric} and Lemma \ref{lem:keyInequality},
it follows that if $x$ has a $d$-block of length $k_{m}$ at the
$n_{m}$-th level, then we have for each $l,m\in\N$ that \[
D_{\Theta}(\xi,\eta_{m})\ll\e^{S_{n_{m}}\chi(x)}\frac{\e^{l\psi((\underline{1}))}+\e^{k_{m}\psi((\underline{d}))}}{\e^{l\phi((\underline{1}))}+\e^{k_{m}\phi((\underline{d}))}}\ll\e^{S_{n_{m}}\chi(x)+\alpha k_{m}}.\]
 Since $\liminf_{m\rightarrow\infty}D_{\Theta}(\xi,\eta_{m})>0$,
it follows that \[
\liminf_{m\rightarrow\infty}\e^{S_{n_{m}}\chi(x)+\alpha k_{m}}>0,\]
 and therefore, \[
\e^{S_{n_{m}}\chi(x)+\alpha k_{m}}\gg1,\,\mbox{for all }\, m\in\N.\]

\end{proof}

\subsection{The upper bound}

We start by observing that \[
\limsup_{n\to\infty}\e^{S_{n}\chi}>0\implies\limsup_{n\to\infty}\frac{S_{n}\phi}{S_{n}\psi}\geq1.\]
 This implies that \[
\dim_{H}\left(\left\{ \limsup_{n\to\infty}\e^{S_{n}\chi}>0\right\} \right)\leq\dim_{H}\left(\left\{ \limsup_{n\to\infty}\frac{S_{n}\phi}{S_{n}\psi}\geq1\right\} \right)=\hat{\beta}\left(1\right).\]
 Here, the final equality holds since the Lyapunov dimension spectrum
$s \mapsto \hat{\beta}\left(s\right)/s$ is decreasing in a neighbourhood
of $1$. Since  $\mathcal{D}_{\infty}$ and $\mathcal{D}_{\sim}\cap\left\{ \limsup_{n\to\infty}\e^{S_{n}\chi}>0\right\} $
are contained in $\left\{ \limsup_{n\to\infty}\e^{S_{n}\chi}>0\right\} $,
the observation above gives the upper bound $\hat{\beta}\left(1\right)$
for the Hausdorff dimension of each of these two sets.

Since $\lim_{n\to\infty}\exp(S_{n}\chi(x))=0$ implies $\pi_{S}(x)\in\mathcal{D}_{\sharp}^{*}$,
except for the countable set of end points of all refinements of the
Markov partition, it is therefore sufficient to show that \[
\dim_{H}\left(\mathcal{D}_{\sim}\cap\mathcal{D}_{\sharp}^{*}\right)\leq-\hat{\beta}\left(-1\right).\]
 Before we come to this, let us first make the following observation,
which also explains why at the end of the previous section we restricted
the discussion to the two cases in (\ref{eq:case1_2}).

\begin{lem}
\label{lem:Alternative} If we are in neither of the two cases in
(\ref{eq:case1_2}), then \[
\mathcal{D}_{\sim}\cap\mathcal{D}_{\sharp}^{*}=\emptyset.\]

\end{lem}
\begin{proof}
Let $x=(x_{1}x_{2}\ldots)\in\Sigma$ be given such that $\xi:=\pi_{S}(x)\in\mathcal{D}_{\sharp}^{*}$.
Let us assume that $x$ has a strict $j$-block of length $k$ at
the $n$-th level, $j\in \{1,d\}$. We have to distinguish two 
cases. The first of these is \[
\frac{\psi\left(\underline{1}\right)}{\phi\left(\underline{1}\right)}\geq1\:\mbox{ and }\frac{\psi\left(\underline{d}\right)}{\phi\left(\underline{d}\right)}\geq1.\]
Then $\e^{\psi((\underline{i}))}\leq\e^{\phi((\underline{i}))}$,
for $i\in \{1,d\}$, and we clearly have $\frac{\e^{l\psi((\underline{1}))}+\e^{k\psi((\underline{d}))}}{\e^{l\phi((\underline{1}))}+\e^{k\phi((\underline{d}))}}\leq1$,
for all $k,l\in\N$. By combining this observation with Proposition \ref{geometric}
and Lemma \ref{geom2}, it follows that $\xi\in\mathcal{D}_{0}$, and hence
$\xi\notin\mathcal{D}_{\sim}$. 

The second case is \[
\frac{\psi\left(\underline{1}\right)}{\phi\left(\underline{1}\right)}=\frac{\psi\left(\underline{d}\right)}{\phi\left(\underline{d}\right)}<1.\]
Similarly to the proof of Lemma \ref{lem:keyInequality}, it then follows that for each $i\in\left\{ 1,d\right\} \setminus\left\{ j\right\} $ and for all $l,k\in\N$, we have \[
\frac{\e^{l\psi((\underline{i}))}+\e^{k\psi((\underline{j}))}}{\e^{l\phi((\underline{i}))}+\e^{k\phi((\underline{j}))}}\leq2\e^{k\chi((\underline{j}))}.\]
 Therefore, it follows that for each $\eta\in\pi_{S}([x_{1}\ldots x_{n-1}]\setminus[x_{1}\ldots x_{n}])$ we have
\[
D_{\Theta}(\xi,\eta)\ll\e^{S_{n+k}\chi\left(x\right)}.\]
Using this observation and Lemma \ref{geom2}, we obtain that the
derivative of $\Theta$ at $\xi$ is equal to $0$, and hence $\xi\notin\mathcal{D}_{\sim}$. 
\end{proof}
We now finally come to the proof of the upper bound for the Hausdorff
dimension of $\mathcal{D}_{\sim}\cap\mathcal{D}_{\sharp}^{*}$. This
part of the proof is inspired by the arguments given in \cite{KessStrat1}.
First note that it is sufficient to show that \[
\dim_{H}(\mathcal{D}_{\sim}\cap\mathcal{D}_{\sharp}^{*})\leq\tilde{\beta}\left(s\right),\:\mbox{for all }s\leq1.\]
 In a nutshell, the idea is to show that for each $s\leq1$ there
is a suitable covering of $\mathcal{D}_{\sim}\cap\mathcal{D}_{\sharp}^{*}$
which then will be used to deduce that the $\tilde{\beta}\left(s\right)$-dimensional
Hausdorff measure of $\mathcal{D}_{\sim}\cap\mathcal{D}_{\sharp}^{*}$
is finite.

For ease of exposition, throughout the remaining part of this section
we will again assume that we are in Case 1 of the two cases in (\ref{eq:case1_2}).
Clearly, the considerations for Case 2 are completely analogous, and
will therefore be omitted. Let us first introduce the stopping time
$\tau_{t}$ with respect to $\chi$ on $\pi_{S}^{-1}(\mathcal{D}_{\sharp}^{*})$
by \[
\tau_{t}(x):=\inf\{k\in\mathbb{N}:S_{k}\chi(x)<-t\},\;\mbox{ for all }\; t>0,\, x\in\pi_{S}^{-1}\left(\mathcal{D}_{\sharp}^{*}\right).\]
 For each $n\in\N$ fix a partition $\mathcal{C}_{n}$ of $\pi_{S}^{-1}(\mathcal{D}_{\sharp}^{*})$
consisting of cylinder sets with the following property: \[
\hbox{For\, each}\,\,[\omega]\in\mathcal{C}_{n}\,\,\hbox{and}\,\, x\in[\omega],\,\hbox{we\, have}\,\,|S_{\tau_{n}\left(x\right)}\chi(x)+n|\ll1.\]

\noindent Moreover, for $\epsilon>0$ we define \[
\mathcal{C}_{n}\left(\epsilon\right):=\left\{ [\omega\underline{d}_{n_{\epsilon}}]:[\omega]\in\mathcal{C}_{n}\right\} ,\]
 where $n_{\epsilon}$ is given by $n_{\epsilon}:=\left\lfloor n\left(1-\epsilon\right)/\alpha\right\rfloor $.
For $s\in\left(0,1\right)$ we choose $\epsilon>0$ such that \[
(1-\epsilon)\cdot\tilde{\beta}\left(s\right)>\left(-\chi\left(\underline{1}\right)/\phi\left(\underline{1}\right)\right)\cdot\beta\left(s\right).\]
 This is possible, since on the one hand we have $\tilde{\beta}\left(s\right)-\beta\left(s\right)=s>0$
and hence $\tilde{\beta}\left(s\right)>\beta\left(s\right)$, for
all $s\in\left(0,1\right)$. On the other hand, the fact that $\psi<0$
immediately implies that $(-\chi\left(\underline{1}\right)/\phi\left(\underline{1}\right))<1$.
Recall that we are assuming that Case 1 of (\ref{eq:case1_2}) holds,
and therefore we have that $(-\chi\left(\underline{1}\right)/\phi\left(\underline{1}\right))=1-\psi\left(\underline{1}\right)/\phi\left(\underline{1}\right)>0$.
It then follows \\
 \\
 ${\displaystyle \sum_{n\in\N}\sum_{C\in\mathcal{C}_{n}(\epsilon)}\left(\diam(C)\right)^{\tilde{\beta}\left(s\right)}\asymp\sum_{n\in\N}\sum_{C\in\mathcal{C}_{n}(\epsilon)}\e^{\sup_{x\in C}\tilde{\beta}\left(s\right)S_{\tau_{n}(x)+n_{\epsilon}}\phi(x)}}$
\begin{eqnarray*}
 & \ll & \sum_{n\in\N}\e^{n(1-\epsilon)\tilde{\beta}\left(s\right)\phi(\underline{d})/\alpha}\sum_{C\in\mathcal{C}_{n}}\e^{\tilde{\beta}\left(s\right)\sup_{x\in C}S_{\tau_{n}(x)}\phi(x)}\\
 & \asymp & \sum_{n\in\N}\e^{n(1-\epsilon)\tilde{\beta}\left(s\right)\phi(\underline{d})/\alpha+n\beta\left(s\right)}\sum_{C\in\mathcal{C}_{n}}\e^{\sup_{x\in C}S_{\tau_{n}(x)}\left(\tilde{\beta}\left(s\right)\phi(x)+\beta\left(s\right)\chi(x)\right)}\\
 & \ll & \sum_{n\in\N}\left(\e^{(1-\epsilon)\tilde{\beta}\left(s\right)\phi(\underline{d})/\alpha+\beta\left(s\right)}\right)^{n}<\infty.\end{eqnarray*}
 Here we have used the Gibbs property \begin{eqnarray*}
\sum_{C\in\mathcal{C}_{n}}\e^{\sup_{x\in C}S_{\tau_{n}(x)}\left(\tilde{\beta}\left(s\right)\phi(x)+\beta\left(s\right)\chi(x)\right)} & = & \sum_{C\in\mathcal{C}_{n}}\e^{\sup_{x\in C}S_{\tau_{n}(x)}\left(s\phi(x)+\beta\left(s\right)\psi(x)\right)}\ll1\end{eqnarray*}
 of the Gibbs measure $\mu_{s}$ and the fact that \[
(1-\epsilon)\tilde{\beta}\left(s\right)\phi(\underline{d})/\alpha+\beta\left(s\right)=(1-\epsilon)\tilde{\beta}\left(s\right)\phi(\underline{1})/\chi(\underline{1})+\beta\left(s\right)<0.\]

Thus, for the limsup-set \[
C_{\infty}(\epsilon):=\{\xi\in\mathcal{U}:\xi\in\pi_{S}\left(\mathcal{C}_{n}(\epsilon)\right)\,\,\hbox{for\, infinitely\, many}\,\, n\in\N\}\]
 we now have \[
\dim_{H}(C_{\infty}(\epsilon))\leq\min_{s\in\left(0,1\right)}\tilde{\beta}\left(s\right)=-\hat{\beta}\left(-1\right).\]
 Hence, it remains to show that\[
\mathcal{D}_{\sim}\cap\mathcal{D}_{\sharp}^{*}\subset C_{\infty}(\epsilon),\,\mbox{ for all $\epsilon>0$}.\]
 For this, let $x\in\Sigma$ be given such that $\xi:=\pi_{S}(x)\in\mathcal{D}_{\sim}\cap\mathcal{D}_{\sharp}^{*}$.
By Proposition \ref{Prop3}, there exist strictly increasing sequences
$(n_{m})_{m\in\N}$ and $(k_{m})_{m\in\N}$ of positive integers such
that $x$ has a $d$-block of length $k_{m}$ at the $n_{m}$-th level
and \[
\e^{S_{n_{m}}\chi(x)+k_{m}\alpha}\gg1,\,\,\hbox{for\, each}\,\, m\in\N.\]
 By setting $\ell(n_{m}):=\lfloor S_{n_{m}}\chi(x)\rfloor$, it follows
$\exp(k_{m})\gg\exp(-\ell(n_{m})/\alpha)$. Hence, for each $\epsilon>0$
and for each $m$ sufficiently large, we have $k_{m}\geq-\ell(n_{m})(1-\epsilon)/\alpha)$.
It follows that $\xi\in C_{\infty}(\epsilon)$, which finishes the
proof of the upper bound.

\subsection{The lower bound}

In this section we show that the Hausdorff dimension of each of the
sets $\mathcal{D}_{\sim}$ and $\mathcal{D}_{\infty}$ is bounded
below by $\tilde{\beta}(s_{0})$. Clearly, combining this with the
results of the previous section will then complete the proof of Theorem
\ref{main}. Let us begin with, by showing that \[
\dim_{H}\left(\mathcal{D}_{\sim}\right)\geq\tilde{\beta}(s_{0}).\]
 Recall that $\mu_{s}$ refers to the equilibrium measure for the
potential $s\phi+\beta(s)\psi$, and that $s_{0}$ is chosen so that
\[
\beta^{\prime}(s_{0})=-\frac{\int\phi\text{d}\mu_{s_{0}}}{\int\psi\text{d}\mu_{s_{0}}}=-1.\]
 This implies that \[
0=\int\psi\text{d}\mu_{s_{0}}-\int\phi\text{d}\mu_{s_{0}}=\int\chi\text{d}\mu_{s_{0}}.\]
 By the the variational principle, we have \[
h(\mu_{s_{0}})+s_{0}\int\phi\text{d}\mu_{s_{0}}+\beta(s_{0})\int\psi\text{d}\mu_{s_{0}}=0,\]
 and hence, \[
\frac{h(\mu_{s_{0}})}{-\int\phi\text{d}\mu_{s_{0}}}=\beta(s_{0})+s_{0}=\tilde{\beta}\left(s_{0}\right).\]
 Since we are in the expanding case, we can use Young's formula (see
\cite{Manning} \cite{Young}) to deduce that $\dim_{H}(\pi_{S}(\mu_{s_{0}}))=\tilde{\beta}\left(s_{0}\right)$.
The lower bound for the Hausdorff dimension of $\mathcal{D}_{\sim}$
now follows from combining Corollary \ref{deri} with the following
lemma.

\begin{lem}
\label{manfred} For $\mu_{s_{0}}$-almost every $x\in\Sigma$ we
have \[
\liminf_{n\rightarrow\infty}\e^{S_{n}\chi(x)}=0\text{ and }\limsup_{n\rightarrow\infty}\e^{S_{n}\chi(x)}=\infty.\]

\end{lem}
\begin{proof}
Note that $\int\chi\text{d}\mu_{s_{0}}=0$. Thus, by the law of the
iterated logarithm \cite{DenkerPhilipp:84} we have that there exists
a constant $C>0$ such that for $\mu_{s_{0}}$-almost all $x\in\Sigma$
we have \[
\liminf_{n\rightarrow\infty}\frac{S_{n}\chi(x)}{\sqrt{n\log\log n}}=-C\]
 and \[
\limsup_{n\rightarrow\infty}\frac{S_{n}\chi(x)}{\sqrt{n\log\log n}}=C.\]
 From this we deduce that for $\mu_{s_{0}}$-almost all $x\in\Sigma$
we have \[
\liminf_{n\rightarrow\infty}\e^{S_{n}\chi(x)}=0\text{ and }\limsup_{n\rightarrow\infty}\e^{S_{n}\chi(x)}=\infty.\]

\end{proof}
Lemma \ref{manfred} implies that $\pi_{S}(x)\in\mathcal{D}_{\sim}$
for $\mu_{s_{0}}$-almost every $x\in\Sigma$, and hence, \[
\dim_{H}(\mathcal{D}_{\sim})\geq\dim_{H}(\pi_{S}(\mu_{s_{0}}))=\tilde{\beta}(s_{0}).\]

Therefore, it remains to show that \[
\dim_{H}(\mathcal{D}_{\infty})\geq\tilde{\beta}(s_{0}).\]
 For this, we consider the set of the equilibrium measures $\{\mu_{s}:s>s_{0}\}$.

\begin{lem}
\label{lem:LyapunovPositiv} For $s>s_{0}$, we have that \[
\int\chi\text{d}\mu_{s}>0.\]

\end{lem}
\begin{proof}
Since $\beta$ is strictly convex, we have that $s>s_{0}$ implies
that $\beta'(s)>-1$. This gives \[
\frac{\int\phi\text{d}\mu_{s}}{\int\psi\text{d}\mu_{s}}=-\beta'(s)<1,\]
 and hence, \[
\int\chi\text{d}\mu_{s}>0.\]

\end{proof}
Lemma \ref{lem:LyapunovPositiv} implies that for $\mu_{s}$-almost
every $x\in\Sigma$ we have (recall that we are assuming that $s>s_{0}$)
\[
\lim_{n\rightarrow\infty}\e^{S_{n}\chi(x)}=\infty.\]
 For the following lemma let us introduce the following notations.
For $x=(x_{1}x_{2}\ldots)\in\Sigma$, $k,n\in\N$ and $i\in\{1,d\}$,
let $k_{n}(x):=k$ if $x$ has an $i$-block of length $k$ at the
$n$-th level, and set $k_{n}(x):=0$ if $x_{n+1}\notin\{1,d\}$.
We then have the following routine Khintchine-type estimate, where
$\kappa_{i,s}:=-(s\phi((\underline{i}))+\beta(s)\psi((\underline{i})))^{-1}>0$
and $\kappa_{s}:=\min\{\kappa_{i,s}:i=1,d\}$.

\begin{lem}
\label{Khint} For $\mu_{s}$-almost every $x\in\Sigma$ we have \[
\limsup_{n\rightarrow\infty}\frac{k_{n}(x)}{\log n}\leq\kappa_{s}.\]

\end{lem}
\begin{proof}
Let $\mathcal{C}_{n}^{*}:=\{[\omega]:\omega\in A^{n}\}$ and recall
that $\sum_{C\in\mathcal{C}_{n}^{*}}\e^{\sup_{x\in C}S_{n}(s\phi+\beta(s)\psi)(x)}\asymp1$,
for all $n\in\N$. For $\epsilon>0$, let $k_{\epsilon,i,n}:=\left\lfloor (1+\epsilon)\kappa_{i,s}\log n\right\rfloor $.
We then have \[
\sum_{n\in\N}\sum_{[x_{1}\ldots x_{n}]\in\mathcal{C}_{n}^{*}}\e^{\sup_{x\in[x_{1}\ldots x_{n}\underline{i}_{k_{\epsilon,i,n}]}}S_{n+k_{\epsilon,i,n}}(s\phi+\beta(s)\psi)(x)}\ll\sum_{n\in\N}n^{-(1+\epsilon)}.\]
 Hence, by the Borel-Cantelli Lemma, we have that the set of elements
in $\Sigma$ which lie in cylinder sets of the form $[x_{1}\ldots x_{n}\underline{i}_{k_{\epsilon,i,n}}]$
for infinitely many $n\in\N$ has $\mu_{s}$-measure equal to zero.
By passing to the complement of this limsup-set, the statement in
the lemma follows. 
\end{proof}
We can now complete the proof of Theorem \ref{main} as follows. By
Lemma \ref{geom1} we have that there exists a constant $c>0$ such
that for each $x=(x_{1}x_{2}\ldots)\in\Sigma$ and for each sequence
$\left(\eta_{n}\right)_{n}$ in $\mathcal{U}$ tending to $\xi:=\pi_{S}(x)$,
\[
\liminf_{n\rightarrow\infty}D_{\Theta}(\xi,\eta_{n})\geq c\cdot\liminf_{n\to\infty}\e^{S_{n}\chi(x)}\e^{-k_{n}\delta}.\]
 Moreover, using Lemma \ref{lem:LyapunovPositiv} and the of $\mu_{s}$,
it follows that for $\mu_{s}$-almost every $x\in\Sigma$ we have
\[
\lim_{n\rightarrow\infty}\frac{1}{n}(S_{n}\chi(x))=\int\chi\text{d}\mu_{s}=:c_{\chi}(x)>0.\]
 Combining this with Lemma \ref{Khint}, it follows that for $\mu_{s}$-almost
every $x\in\Sigma$, with $\xi=\pi_{S}(x)$, we have \[
\liminf_{n\rightarrow\infty}D_{\Theta}(\xi,\eta_{n})\geq c\,\liminf_{n\to\infty}\e^{S_{n}\chi(x)}\,\e^{-\delta k_{n}(x)}\geq c\,\liminf_{n\to\infty}\e^{nc_{\chi}(x)}\, n^{-\delta \kappa_{s}}=\infty.\]
 This implies \[
\lim_{n\rightarrow\infty}D_{\Theta}(\xi,\eta_{n})=\infty,\quad\mu_{s}\mbox{-almost everywhere}.\]
 Since $\pi_{S}$ is bijective except on a countable number of points,
we now conclude that for all $s>s_{0}$ we have \[
\dim_{H}(\mathcal{D}_{\infty})\geq\dim_{H}(\pi_{S}(\mu_{s}))=-\hat{\beta}\left(\beta'\left(s\right)\right)/\beta'\left(s\right).\]
 To complete the proof, simply note that $-\hat{\beta}\left(\beta'\left(s\right)\right)/\beta'\left(s\right)\nearrow\tilde{\beta}(s_{0})$,
for $s\searrow s_{0}.$ This finishes the proof of Theorem \ref{main}.

\section{\label{sec:Proof-of-Theorem2}Proof of Theorem \ref{thm:Rigidity}}

If $\left(\mathcal{U},S\right)$ and $\left(\mathcal{U},T\right)$
are $C^{1+\epsilon}$ conjugate, then we clearly have that $\mathcal{D}_{\sim}\left(S,T\right)=\emptyset$,
and hence $\dim_{H}\left(\mathcal{D}_{\sim}\left(S,T\right)\right)=0$.
This gives one direction of the equivalence in Theorem \ref{thm:Rigidity}.\\
 For the other direction, assume that $\dim_{H}\left(\mathcal{D}_{\sim}\left(S,T\right)\right)=0$.
Then Theorem \ref{main} implies that $\phi$ and $\psi$ are cohomologically
dependent. That is, there exist $b,c\in\R\setminus\{0\}$ and a Hölder
continuous function $u:\Sigma\to\R$ such that \[
b\phi+c\psi=u-u\circ\sigma.\]
 We then have for all $s\in\R$ that $P\left(s\phi-b/c\beta\left(s\right)\phi\right)=P\left(\left(s-b/c\beta\left(s\right)\right)\phi\right)=0$,
and hence $\beta\left(s\right)=(s-1)c/b$. Combining this with $\beta\left(0\right)=1$,
it follows that $b/c=-1$, and therefore,\[
\psi-\phi=\chi=v-v\circ\sigma,\]
 for some Hölder continuous function $v:\Sigma\to\R$. Note that we
now in particular also have that $\psi\left(\left(\underline{i}\right)\right)=\phi\left(\left(\underline{i}\right)\right)$,
for each $i\in\left\{ 1,d\right\} $. Combining this with Proposition
\ref{geometric}, it follows that uniformly for all $\xi,\eta\in\mathcal{U}$
we have \[
D_{\Theta}(\xi,\eta)\asymp1.\]
 This shows that there exists a constant $c_{0}>1$ such that for
all $\xi\in\mathcal{U}$ we have \[
c_{0}^{-1}<\liminf_{\eta\to\xi}D_{\Theta}(\xi,\eta)\leq\limsup_{\eta\to\xi}D_{\Theta}(\xi,\eta)<c_{0}.\]
 Since the derivative of $\Theta$ exists Lebesgue-almost everywhere,
it follows that for Lebesgue-almost every $\xi\in\mathcal{U}$ we
have that $\Theta'\left(\xi\right)$ is uniformly bounded away from
zero and infinity. We can now complete the proof by arguing similar
as in \cite{Guizhen} as follows (see the introduction for a statement
of the main result of \cite{Guizhen}). We have split the discussion
into four steps. Here, for $c\in\R$, we let $f_{c}:\R\to\R$ denote
the multiplication map given by $x\mapsto c\cdot x$, and we have
put $\sigma_{0}:=S'\left(0\right)=T'\left(0\right)$. Note that, since
$\psi\left(\left(\underline{1}\right)\right)=\phi\left(\left(\underline{1}\right)\right)$,
we clearly have that $S'\left(0\right)=T'\left(0\right)$.

\emph{Linearisation}: For each $n\in\N$, let $S_{1}^{-n}$ and $T_{1}^{-n}$
denote the inverse branches of $S^{n}$ and $T^{n}$ respectively,
such that $0$ is contained in $S_{1}^{-n}(\mathcal{U})$ and $T_{1}^{-n}(\mathcal{U})$.
 Using the bounded distortion property and the fact that $(S_{1}^{-n})'$
and $(T_{1}^{-n})'$ are uniformly Hölder continuous, we have, by
Arzelà-Ascoli, that there exist subsequences of $\left(f_{\sigma_{0}^{n}}\circ S_{1}^{-n}\right)_{n\in\N}$
and $\left(f_{\sigma_{0}^{n}}\circ T_{1}^{-n}\right)_{n\in\N}$ which
converge uniformly on $\mathcal{U}$ to $C^{1+\epsilon}$ diffeomorphisms
$\gamma_{S}$ and $\gamma_{T}$ respectively. Note that we clearly
have that $\gamma_{S}\circ S=f_{\sigma_{0}}\circ\gamma_{S}$ and $\gamma_{T}\circ T=f_{\sigma_{0}}\circ\gamma_{T}$.

\emph{Differentiation}: The uniform Hölder continuity of $(S_{1}^{-n})'$
and $(T_{1}^{-n})'$ and the fact that the conjugacy $\Theta$ is
bi-Lipschitz imply that the right derivative of $\Theta$ at zero
exists and that it has a finite and positive value.

\emph{Localisation}: We have that \begin{eqnarray*}
f_{\sigma_{0}}\circ\gamma_{T}\circ\Theta\circ\gamma_{S}^{-1} & = & \gamma_{T}\circ T\circ\Theta\circ\gamma_{S}^{-1}=\gamma_{T}\circ\Theta\circ S\circ\gamma_{S}^{-1}\\
 & = & \gamma_{T}\circ\Theta\circ\gamma_{S}^{-1}\circ\gamma_{S}\circ S\circ\gamma_{S}^{-1}=\gamma_{T}\circ\Theta\circ\gamma_{S}^{-1}\circ f_{\sigma_{0}}\circ\gamma_{S}\circ\gamma_{S}^{-1}\\
 & = & \gamma_{T}\circ\Theta\circ\gamma_{S}^{-1}\circ f_{\sigma_{0}},\end{eqnarray*}
 which shows that $\gamma_{T}\circ\Theta\circ\gamma_{S}^{-1}$ commutes
with $f_{\sigma_{0}}$. Using this and the differentiability of $\Theta$
at $0$, we now obtain on the domain of $\gamma_{S}$ that $\gamma_{T}\circ\Theta\circ\gamma_{S}^{-1}=(f_{\sigma_{0}})^{n}\circ\gamma_{T}\circ\Theta\circ\gamma_{S}^{-1}\circ(f_{1/\sigma_{0}})^{n}$.
Therefore, we now have for $\xi$ in this domain \[
\gamma_{T}\circ\Theta\circ\gamma_{S}^{-1}(\xi)=\frac{\gamma_{T}\circ\Theta\circ\gamma_{S}^{-1}(\sigma_{0}^{-n}\cdot\xi)}{\sigma_{0}^{-n}\cdot\xi}\cdot\xi\to\kappa_{0}\cdot\xi,\,\mbox{ for $n$ tending to infinity},\]
 where $\kappa_{0}>0$ denotes the right derivative of $\gamma_{T}\circ\Theta\circ\gamma_{S}^{-1}$
at zero. It now follows that there exists $\delta>0$ such that $\Theta|_{\left[0,\delta\right]}$
is a $C^{1+\epsilon}$ diffeomorphism.

\emph{Globalisation}: Let $n\in\N$ be chosen such that $S_{1}^{-n}\left(\mathcal{U}\right)\subset\left[0,\delta\right]$.
Since $\Theta=T^{n}\circ\Theta\circ S_{1}^{-n}$, it follows that
$\Theta:\mathcal{U}\to\mathcal{U}$ is a $C^{1+\epsilon}$ diffeomorphism.
This completes the proof of the main part of Theorem \ref{thm:Rigidity}.

In order to prove the Hölder regularity of $\Theta$, as claimed in
part (2) of Theorem \ref{thm:Rigidity}, let $\xi,\eta\in\mathcal{U}$
be given and put $\rho:=\left(\sup_{s\in\R}-\beta'(s)\right)>0$.
Clearly, we then have $\frac{S_{n}\psi\left(x\right)}{S_{n}\phi\left(x\right)}>1/\rho$,
for all $x\in\Sigma$ and $n\in\N$. Without loss of generality, we
can assume that $\xi=\pi_{S}\left(x_{1}x_{2}\ldots\right)$$<\eta=\pi_{S}\left(y_{1}y_{2}\ldots\right)$
and $\eta\in\pi_{S}\left(\left[x_{1}\ldots x_{n}\right]\setminus\left[x_{1}\ldots x_{n+1}\right]\right)$.
Moreover, let us only consider the case where $x$ has a strict $d$-block
of length $k$ at the $n$-th level and $y$ has a strict $1$-block
of length $l$ at the $\left(n+1\right)$-th level, for some $k,l,n\in\N$.
Then there exists a uniform constant $C>0$ such that\begin{eqnarray*}
\left|\Theta\left(\xi\right)-\Theta\left(\eta\right)\right| & \leq & C\left(\e^{S_{n+k}\psi\left(x\right)}+\e^{S_{n+l}\psi\left(y\right)}\right)\\
 & = & C\left(\e^{\frac{S_{n+k}\psi\left(x\right)}{S_{n+k}\phi\left(x\right)}S_{n+k}\phi\left(x\right)}+\e^{\frac{S_{n+l}\psi\left(y\right)}{S_{n+l}\phi\left(y\right)}S_{n+l}\phi\left(y\right)}\right)\\
 & \leq & C\left(\e^{\rho^{-1}S_{n+k}\phi\left(x\right)}+\e^{\rho^{-1}S_{n+l}\phi\left(y\right)}\right)\\
 & \leq & 2C\left(\e^{S_{n+k}\phi\left(x\right)}+\e^{S_{n+l}\phi\left(y\right)}\right)^{1/\rho}\\
 & \leq & 2C\left|\xi-\eta\right|^{1/\rho}.\end{eqnarray*}
Note that in the case in which one of the blocks has infinite word
length, then one has to use approximations of this block by words
of finite lengths. 

It remains to show that if $S$ and $T$ are cohomologically independent,
then $\Theta$ has to be singular with respect to the Lebesgue measure
$\lambda$. For this note that on the unit interval without the boundary
points of all refinements of the Markov partition we have that $\Theta=\pi_{T}\circ\pi_{S}^{-1}$.
Therefore, it is sufficient to show that the measure $\lambda\circ\Theta$,
whose distribution function is equal to $\Theta$, is singular with
respect to $\lambda$. Since $\mu_{\psi}\circ\pi_{T}^{-1},\mu_{\phi}\circ\pi_{S}^{-1}$
and $\lambda$ are all in the same measure class, it follows that
$\lambda\circ\Theta$ is absolutely continuous to $\mu_{\psi}\circ\pi_{T}^{-1}\circ\pi_{T}\circ\pi_{S}^{-1}=\mu_{\psi}\circ\pi_{S}^{-1}$.
On the other hand, since $S$ and $T$ are cohomologically independent,
$\mu_{\psi}\circ\pi_{S}^{-1}$ is singular with respect to $\mu_{\phi}\circ\pi_{S}^{-1}$
. This finishes the proof of Theorem \ref{thm:Rigidity}.

\section{\label{sec:Examples-and-proof}Examples}

In this section we consider two families of examples: The Salem family
and the sine family. For the Salem family we will see in Section \ref{sec:C^{k-2}-dependence}
that it gives rise to conjugacies whose sets of non-differentiability
have Hausdorff dimensions arbitrarily close to zero.

\textbf{Example 1} (\emph{The Salem Family}): 
Let us consider a class of examples which was studied by Salem in
\cite{Salem}. Namely, we consider the family of conjugacy maps $\left\{ \Theta_{\tau}:\tau\in(0,1)\backslash\left\{ \frac{1}{2}\right\} \right\} $
which arises from the following endomorphisms of $\mathcal{U}$. For
$\xi\in\mathcal{U}$, we define \[
T(\xi):=2\,\xi\mod1\,\mbox{ and }\, S_{\tau}(\xi):=\left\{ \begin{array}{ccc}
\xi/\tau & \text{ if } & 0\leq\xi\leq\tau\\
(\xi-\tau)/(1-\tau) & \text{ if } & \tau<\xi\leq1.\end{array}\right.\]

\begin{figure}[h]
 \subfigure[$\tau=1/5$]{\includegraphics[width=0.48\textwidth]{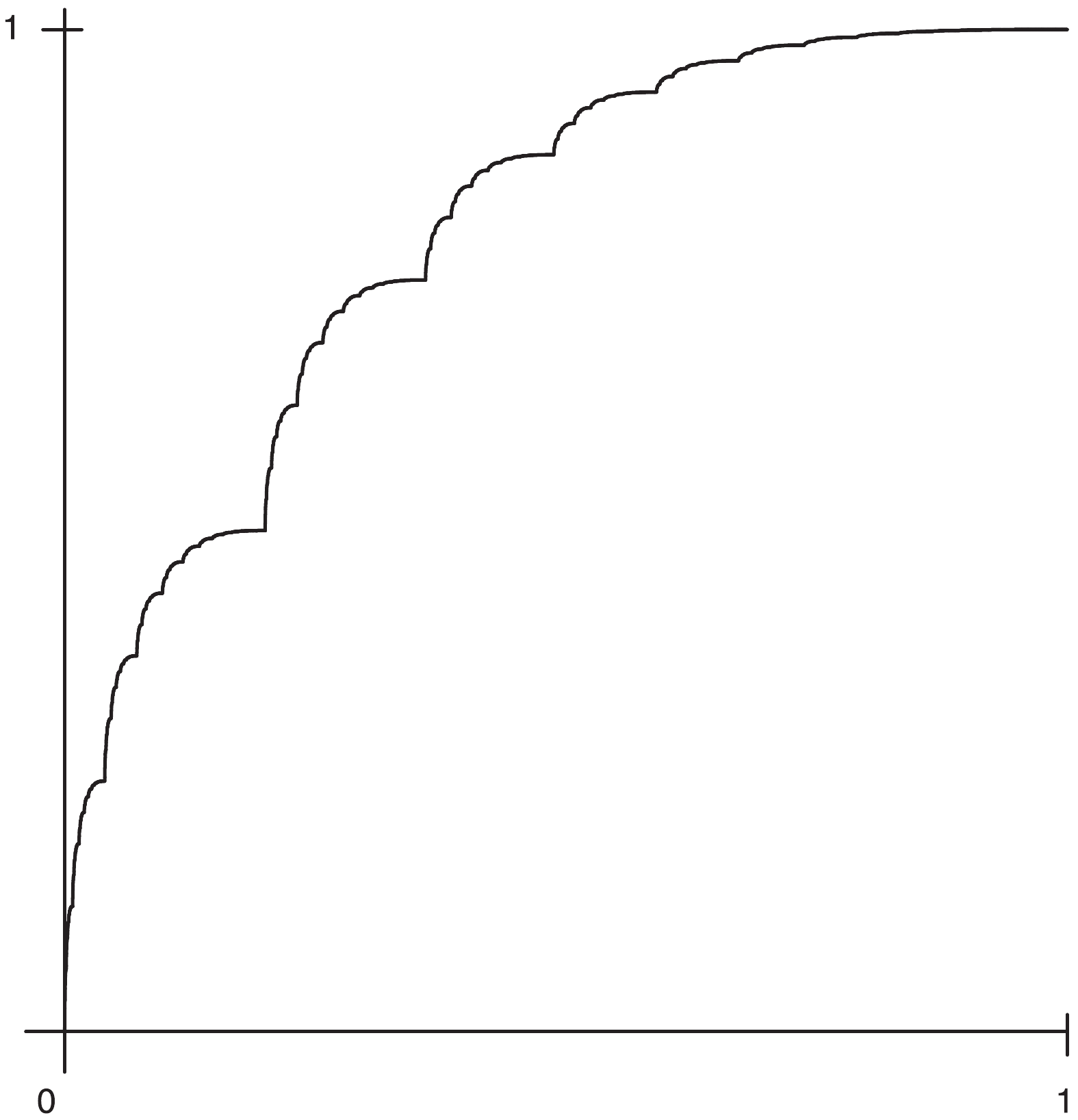}}~~\subfigure[$\tau=2/5$]{\includegraphics[width=0.48\textwidth]{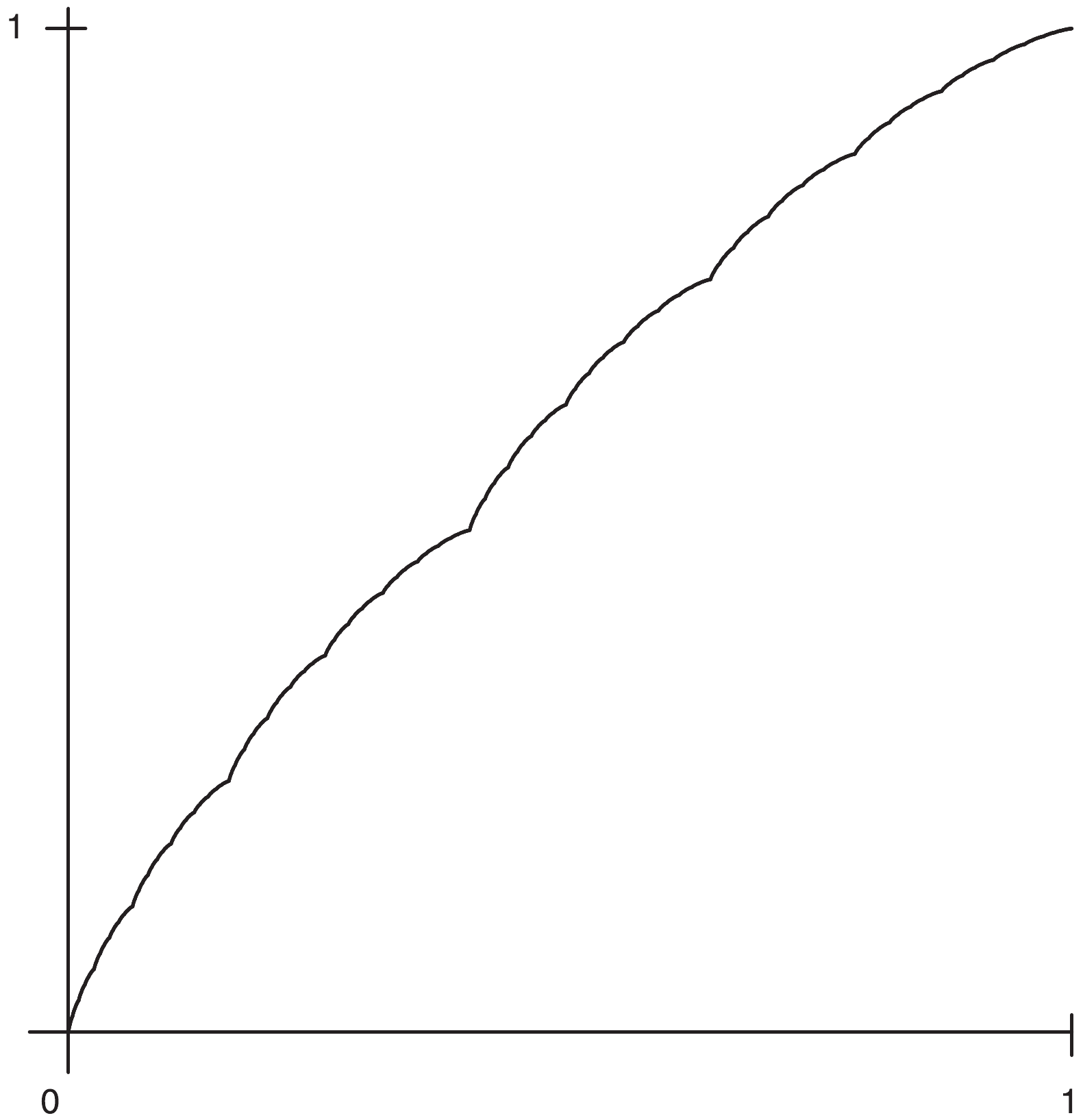}}

\caption{\label{fig:Conjugacy-maps1}The conjugating map $\Theta_{\tau}$ for
the Salem case.}

\end{figure}

 The maps $\Theta_{\tau}:[0,1]\rightarrow[0,1]$ are then given by
$T\circ\Theta_{\tau}=\Theta_{\tau}\circ S_{\tau}$. One immediately
verifies that $\Theta_{\tau}$ is strictly monotone and has the property
that $\Theta_{\tau}'(\xi)=0$ for Lebesgue-almost every $\xi\in\mathcal{U}$.
Note that the conjugacies considered in \cite{Salem} are in fact
dual to the ones which we consider here. However, this has no effect
on the Hausdorff dimension of $\mathcal{D}_{\sim}(S_{\tau},T)$ (see
Remark \ref{r:1} (2)), and our conjugacies have the advantage that
they allow us to determine $\beta_{\tau}$ and $\dim_{H}(\mathcal{D}_{\sim}(S_{\tau},T))$
rather explicitly. For this first note that in the current situation
the potential functions $\phi_{\tau}$ and $\psi$ are given for $x=(x_{1}x_{2}\ldots)\in\Sigma$
by \[
\psi(x)=-\log2\,\mbox{ and }\,\phi_{\tau}(x)=\left\{ \begin{array}{lll}
\log\tau & \text{ if } & x_{1}=1\\
\log(1-\tau) & \text{ if } & x_{1}=2.\end{array}\right.\]
 The function $\beta_{\tau}$ is defined implicitly by $P(s\phi_{\tau}+\beta(t)\psi)=0$.
Since $\exp(s\log\tau-\beta_{\tau}(s)\log2)+\exp(s\log(1-\tau)-\beta_{\tau}(s)\log2)=1$,
an elementary calculation gives that $\beta_{\tau}$ is given explicitly
by \[
\beta_{\tau}(s)=\frac{P(s\phi_{\tau})}{\log2}=\log_{2}(\tau^{s}+(1-\tau)^{s}),\,\mbox{ for each }\, s\in\R.\]
 In order to compute $\dim_{H}(\mathcal{D}_{\sim}(S_{\tau},T))$,
let $\nu_{\tau}$ be the $(p_{\tau},1-p_{\tau})$-Bernoulli measure
such that $\int\psi\text{d}\nu_{\tau}/\int\phi_{\tau}\text{d}\nu_{\tau}=1$.
We then have that \[
1=\frac{\int\psi\text{d}\nu_{\tau}}{\int\phi_{\tau}\text{d}\nu_{\tau}}=-p_{\tau}\log_{2}\tau-(1-p_{\tau})\log_{2}(1-\tau),\]
 and hence, \[
p_{\tau}=\frac{1+\log_{2}(1-\tau)}{\log_{2}\tau-\log_{2}(1-\tau)}.\]
 One then immediately verifies that the supremum in Remark \ref{r:1}
(2) is attained for $\mu=\nu_{\tau}$, and hence it follows that \[
\dim_{H}(\mathcal{D}_{\sim}(S_{\tau},T))=-p_{\tau}\log_{2}p_{\tau}-(1-p_{\tau})\log_{2}(1-p_{\tau}).\]
 The graphs of $\beta_{\tau}$ and of the corresponding dimension
spectrum are given in Fig. \ref{SalemBetaDim}. Also, Fig \ref{fig:AnalyticDepend}
(b) shows $\dim_{H}(\mathcal{D}_{\sim}(S_{\tau},T))$ in dependence
on $\tau$. \\
\begin{figure}[h]
\psfrag{1}{ $1$} \psfrag{a}{ $s$} \psfrag{dimH}{ $\dim_{H}(\mathcal{L}(s))$}
\psfrag{dimH1}{ $\dim_{H}(\mathcal{D}_{\sim})$}\subfigure[The $\beta$--graph]{\includegraphics[width=0.48\textwidth]{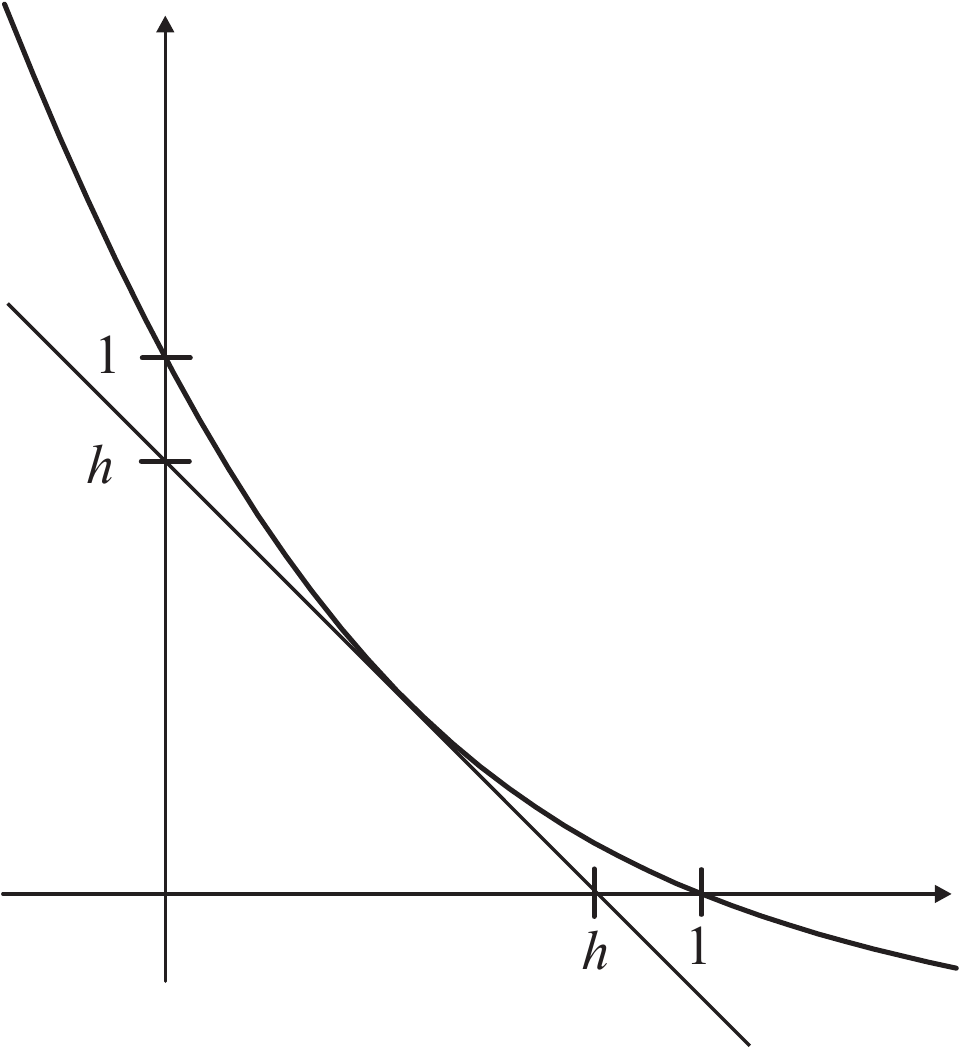}}~~
\subfigure[The Lyapunov spectrum $s \mapsto\dim_H(\mathcal{L}(s))$.]{\includegraphics[width=0.5\textwidth]{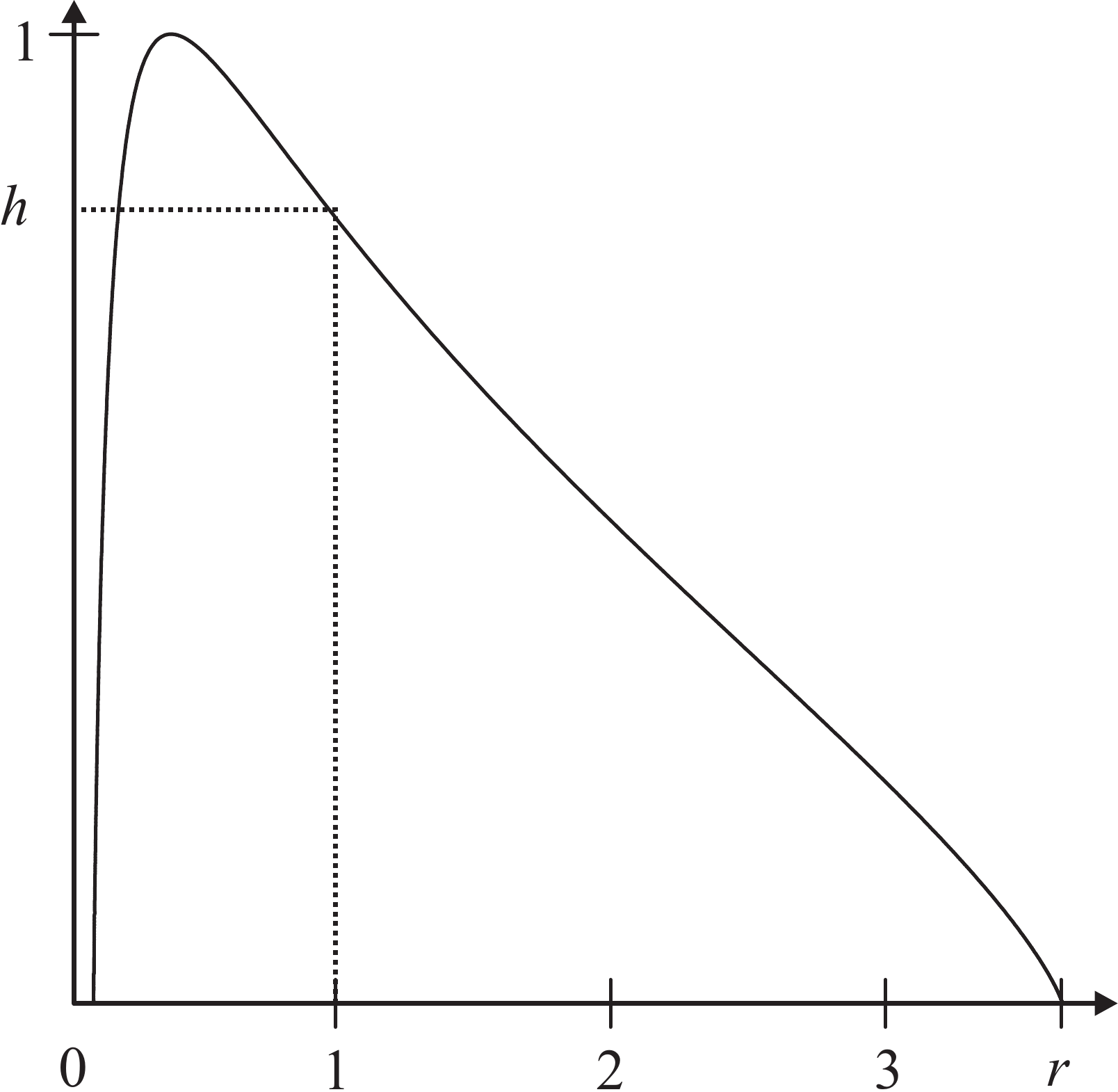}} 

\caption{\label{SalemBetaDim}The $\beta$-graph and the graph of the Lyapunov
dimension spectrum in the Salem case for $\tau=0.08$; in both figures
$h=0.8107...$ denotes the Hausdorff dimension of $\mathcal{D}_{\sim}$.
The conjugacy $\Theta_{\tau}$ is $1/r$-Hölder regular.}

\end{figure}

Finally, let us mention that one can also explicitly calculate the
number $s_{0}(S_{\tau})$ which is determined by $\beta_{\tau}'(s_{0}
(S_{\tau}))=-1$.
A straight forward calculation gives that \[
s_{0}(S_{\tau})=\left(\log(\tau^{-1}-1)\right)^{-1}\log\left(\frac{\log(2\tau)}{\log(2/(1-\tau))}\right).\]
 
\begin{figure}[h]
\subfigure[The graph of $\tau\rightarrow\dim D_{\sim}(R_{\tau},T)$ for the sine family with parameter $\tau$ .]{\includegraphics[width=0.49\textwidth]{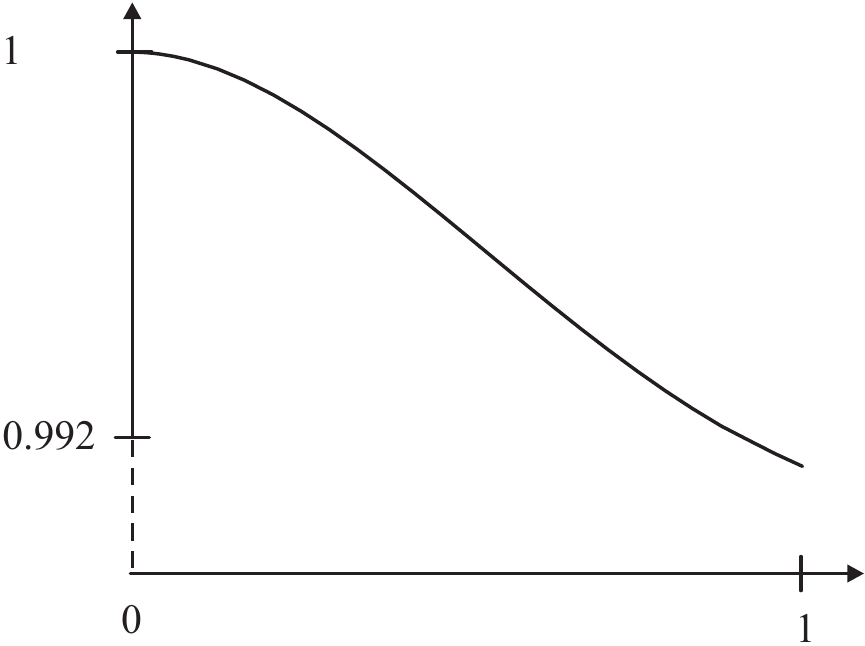}}~\subfigure[The graph of $\tau\rightarrow\dim D_{\sim}(S_{\tau},T)$ for the Salem family with parameter $\tau$ .]{\includegraphics[width=0.48\textwidth]{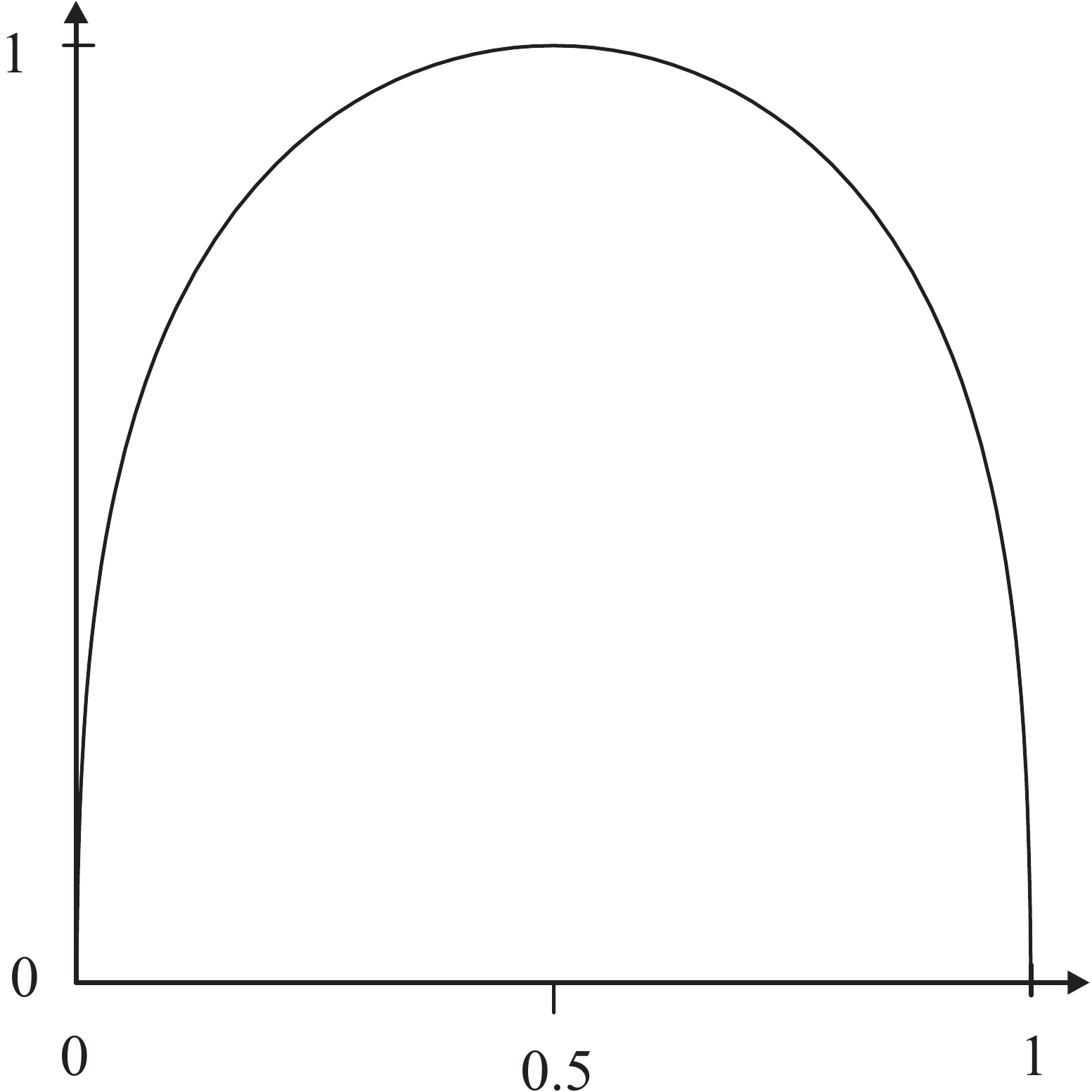}}

\caption{\label{fig:AnalyticDepend}The two dimension spectra. }

\end{figure}

\textbf{Example 2} (\emph{The Sine Family}): Let $T$ be given as
in the previous example, and for each $\tau\in(0,1)$ let the map
$R_{\tau}:\mathcal{U}\to\mathcal{U}$ be defined by \[
R_{\tau}(\xi):=2\xi+\frac{\tau}{2\pi}\sin(2\pi\xi)\mod1,\,\mbox{ for each }\,\xi\in\mathcal{U}.\]
 The associated conjugacies $\Psi_{\tau}$ are then given by $\Psi_{\tau}\circ R_{\tau}=T\circ\Psi_{\tau}$
(see Fig. \ref{fig:Conjugacy-maps}). We can then use Theorem \ref{main}
to compute the Hausdorff dimension of the set $\mathcal{D}_{\sim}(R_{\tau},T)$
of points at which $\Psi_{\tau}$ is not differentiable in the generalised
sense. This is plotted as a graph in Fig \ref{fig:AnalyticDepend}.
(Note that taking the conjugacy in the other direction would yield
exactly the same result).

\begin{figure}[h]
 \subfigure[$\tau=0.4$]{\includegraphics[width=0.45\textwidth]{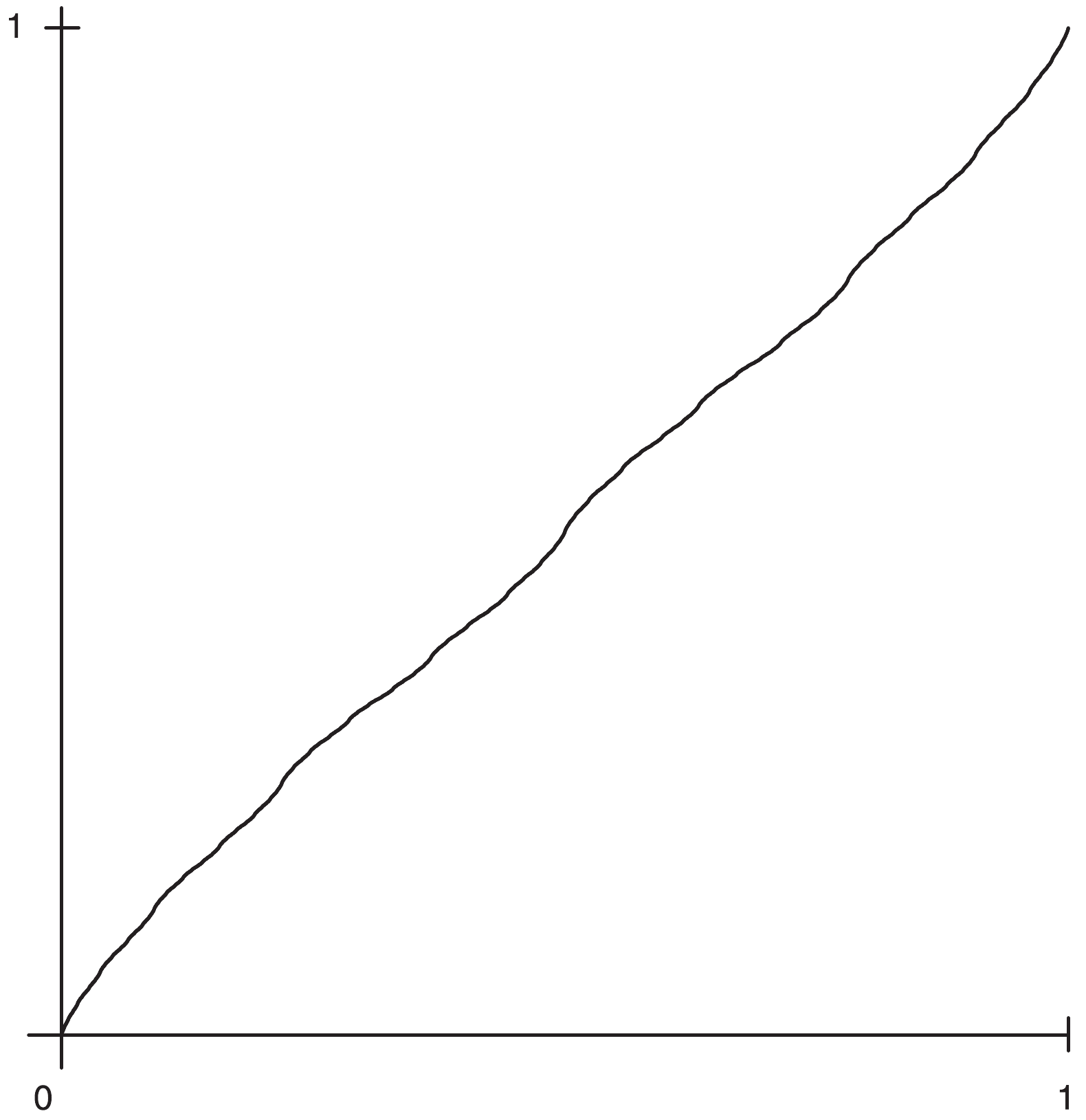}}~~
~\subfigure[$\tau=0.8$]{\includegraphics[width=0.45\textwidth]{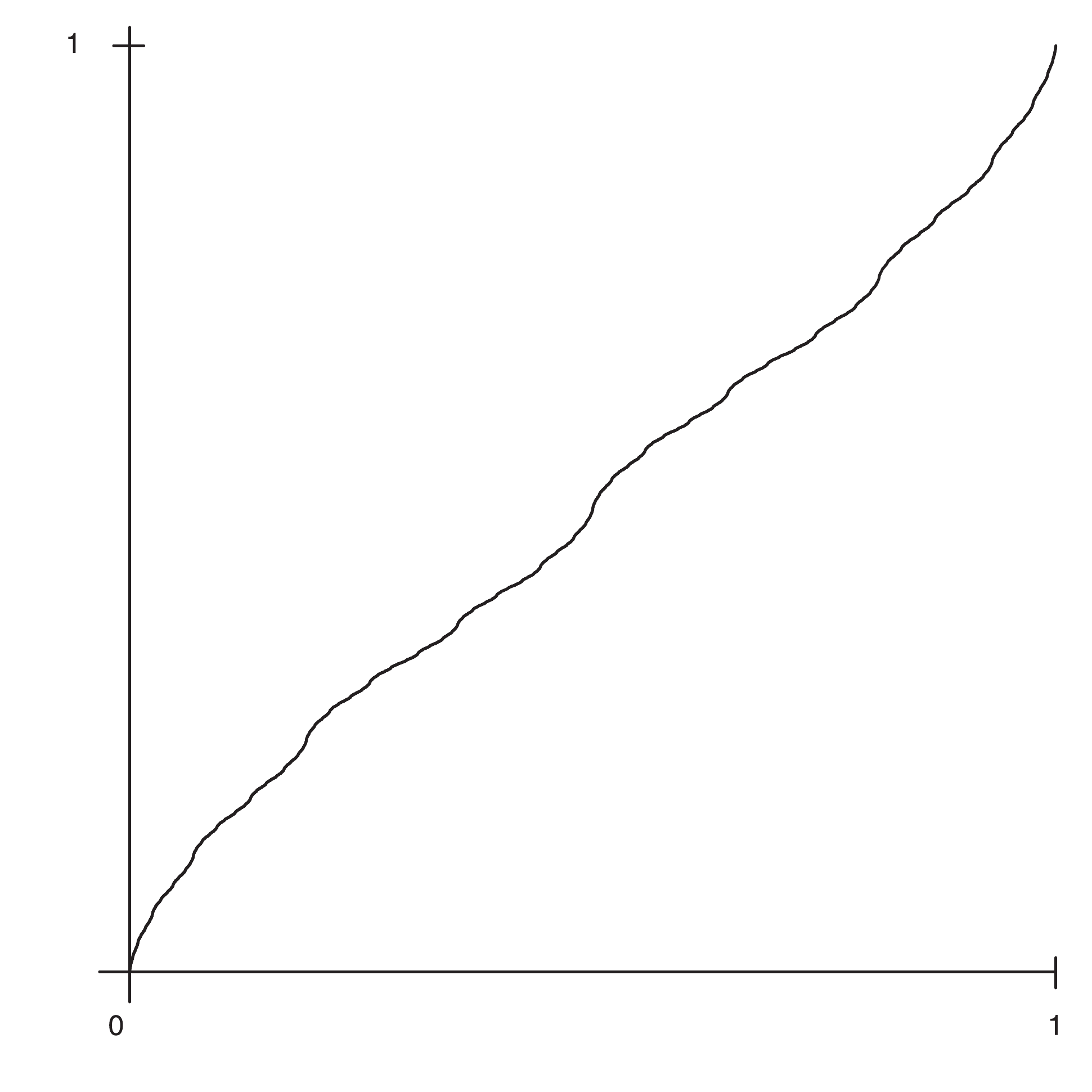}}

\caption{\label{fig:Conjugacy-maps}The graphs of the conjugating maps $\Psi_{\tau}$
for $\tau=0.4$ and $\tau=0.8$ in the sine-family example.}

\end{figure}

\section{\label{sec:C^{k-2}-dependence} Proofs of Propositions \ref{prop1}
and \ref{prop2}}

{\em Proof of Proposition \ref{prop1}}: We start by observing
that the Hausdorff dimension of the set $\mathcal{D}_{\sim}$ depends
regularly on the expanding maps. Let $T_{\tau}:\mathcal{U}\to\mathcal{U}$
be elements of the Banach manifold of the $C^{k}$-family of expanding
maps, with a $C^{k}$ dependence on $\tau\in(-\epsilon,\epsilon)$,
say, and assume that $T_{0}$ is the usual $d$-to-$1$ linear expanding
map. Let $0=a_{0}^{(\tau)}<a_{1}^{(\tau)}<\cdots a_{d-1}^{(\tau)}<a_{d}^{(\tau)}=1$
denote the $T_{\tau}$-preimages of zero. For each $\alpha>0$, we
then define the operator $\mathcal{T}_{\tau}:C^{\alpha}(\mathcal{U},\R)\to C^{\alpha}(\mathcal{U},\R)$
on the space of $\alpha$-Hölder continuous functions (see e.g. \cite{KH})
by \[
(\mathcal{T}_{\tau}h)(\xi):=\frac{1}{d}h\left(\{T_{\tau}(\xi)\}\right)+\frac{j}{d},\,\text{ for each }\xi\in[a_{j}^{(\tau)},a_{j+1}^{(\tau)}],j\in\{0,...,d-1\}.\]
 Also, with $\Vert h\Vert_{\infty}$ denoting the usual supremum norm,
we define a norm $\Vert\cdot\Vert$ on $C^{\alpha}(\mathcal{U},\R)$
by \[
\Vert h\Vert:=\sup_{\xi\neq\eta}\frac{|h(\xi)-h(\eta)|}{|\xi-\eta|^{\alpha}}+\Vert h\Vert_{\infty}.\]
 We observe that on each of the intervals $[a_{j}^{(\tau)},a_{j+1}^{(\tau)}]$
we have that \begin{eqnarray*}
|\mathcal{T}_{\tau}h_{1}(\xi)-\mathcal{T}_{\tau}h_{2}(\xi)| & \leq & \frac{1}{d}|h_{1}(T_{\tau}(\xi))-h_{2}(T_{\tau}(\xi))|\\
 & \leq & \frac{1}{d}\Vert h_{1}-h_{2}\Vert_{\infty}\end{eqnarray*}
 and \begin{eqnarray*}
|\mathcal{T}_{\tau}(h_{1}-h_{2})(\xi) & - & \mathcal{T}_{\tau}(h_{1}-h_{2})(\eta)|\\
 & \leq & \frac{1}{d}\Vert h_{1}-h_{2}\Vert_{C^{\alpha}}|T_{\tau}(\xi)-T_{\tau}(\eta)|^{\alpha}\\
 & \leq & \frac{1}{d}\Vert h_{1}-h_{2}\Vert_{C^{\alpha}}|T_{\tau}(\xi)-T_{\tau}(\eta)|^{\alpha}\\
 & \leq & \left(\frac{1}{d}\Vert h_{1}-h_{2}\Vert_{C^{\alpha}}\,\Vert T_{\tau}\Vert_{C^{1}}^{\alpha}\right)|\xi-\eta|^{\alpha}.\end{eqnarray*}
 In particular, for $\alpha>0$ sufficiently small we have that $\mathcal{T}_{\tau}$
is a contraction with respect to $\Vert\cdot\Vert$. Moreover, $(I-\mathcal{T}_{\tau}):C^{\alpha}(\mathcal{U},\R)\to C^{\alpha}(\mathcal{U},\R)$
is invertible, and by the Implicit Function Theorem there exists a
$C^{k}$ family $\left\{ h_{\tau}\in C^{\alpha}(\mathcal{U},\mathcal{U}):\tau\in(-\epsilon,\epsilon)\right\} $
such that $h_{0}$ is the identity map and $\mathcal{T}_{\tau}h_{\tau}=h_{\tau}$.

Let us consider the map $H_{\tau}:(-\epsilon,\epsilon)\to C^{k-1}(\mathcal{U})\times C^{\alpha}(\mathcal{U})$
given by $H_{\tau}(\tau):=(\log|T_{\tau}'|,h_{\tau})$. Clearly, this
map is $C^{k-1}$ as a map on Banach spaces. Also, we define the composition
operator $\mathcal{O}:C^{k-1}(\mathcal{U})\times C^{\alpha}(\mathcal{U})\to C^{\alpha}(\mathcal{U})$
by $\mathcal{O}(f,g):=f\circ g$, which is $C^{k-2}$, by a result
of \cite{Llave} .  We then consider the image of $H_{\tau}$ under $\mathcal{O}$, that
is \[
\mathcal{O}\circ H_{\tau}:\tau\mapsto(\log|T_{\tau}'|,h_{\tau})\mapsto\mathcal{O}(\log|T_{\tau}'|,h_{\tau})=\log|T_{\tau}'|\circ h_{\tau}\in C^{\alpha},\]
 which is again $C^{k-2}$ \cite{Llave}. (Note that if instead we would consider $\tilde{\mathcal{O}}:C^{k-1}(\mathcal{U})\times C^{0}(\mathcal{U})\to C^{0}(\mathcal{U})$,
then $\tilde{\mathcal{O}}(H_{\tau})$ would be $C^{k-1}$; but we
need to work with Hölder functions, which causes the loss of an extra
derivative.) \\
 Now consider the potential function $\phi_{\tau}$, given by $\phi_{\tau}(\xi):=-\log|T_{\tau}'(h_{\tau}(\pi_{T_{\tau}}(\xi)))|$
for $\xi\in\mathcal{U}$, and then let $\beta_{\tau}(s)$ be defined
implicitly by \[
P(-s\log|T'_{\tau}|+\beta_{\tau}(s)\log|T'_{0}|)=0.\]
 Since the pressure function is analytic, the Implicit Function Theorem
implies that the function given by $\tau\mapsto\beta_{\tau}$ is analytic.
Also, it follows that the function given by $\tau \mapsto \dim_H(\mathcal{D}_{\sim}(T_{\tau},T_{0}))$
is a $C^{k-2}$ function (for an example see Fig. \ref{fig:AnalyticDepend}).
This completes the proof of Proposition 1.4 \\

{\em Proof of Proposition \ref{prop2}}: The aim is to show that
there exists a conjugacy between two elements of the space 
$C^{2}({\mathbb S}^{1})$ of $C^{2}$ expanding circle maps
such that the Hausdorff dimension of the set of points at which this
conjugacy is non differentiability in the generalised sense is arbitrarily close to $0$. We
start by considering the Salem case but where the maps are defined
on the circle $\mathbb{S}^{1}$. For ease of exposition,  we use the same notation and let  $T:\mathbb{S}^{1}\rightarrow\mathbb{S}^{1}$
and $S_{\tau}:\mathbb{S}^{1}\rightarrow\mathbb{S}^{1}$ refer to the circle maps 
which correspond to the interval maps defined in Example 1. The corresponding conjugacy $\Theta_\tau$ is given as before by $T\circ\Theta_\tau = \Theta_\tau \circ S_{\tau}$.
From our analysis in Example 1 it is clear that   $\dim_{H}(\mathcal{D}_{\sim}(S_{\tau},T))$ tends to zero for  $\tau$ tending to zero (see Fig. 2). However, whereas $T$
is a $C^{2}$ map of the circle, $S_{\tau}$ is clearly not (although, 
it is always piecewise expanding $C^{2}$ when viewed as a map of 
$\mathcal{U}$ into itself).  So, in 
order to find a $C^{2}$ example, we have to
apply some suitable perturbations to $S_{\tau}$. For this, let 
$\beta_\tau$ and $\psi, \phi_{\tau}:\Sigma\rightarrow\R$ be 
given as in Example 1. 
As before we choose $s_{0}(S_\tau)$ satisfying 
$\beta_{\tau}'(s_{0}(S_\tau))=-1$. For the remaining part of the proof, 
let  $\tau \in (0,1)\setminus \{\frac{1}{2}\}$ be fixed. 

\emph{$C^{2}$ denseness}: We use the  metric $d_{\mathrm{KL}}$ considered by
 Keller and Liverani in \cite{KL}. This metric is given, for  expanding
 piecewise 
 $C^{2}$ maps $F$ and $G$ of the unit interval $\mathcal{U}$ into 
 itself, by \begin{eqnarray*}
d_{\mathrm{KL}}(F,G) & := & \inf\{\gamma>0|\exists 
X \subseteq\mathcal{U} \text{ }\exists \,  
\text{ a diffeomorphism } H:\mathcal{U}
\rightarrow\mathcal{U} \text{ such that }\\
 &  & \;\;\;\;\;\;\lambda(X)>1-\gamma,G\left|_{X}=F \circ H\right|_{X}\text{ and 
 }\\
 &  & \;\;\;\;\;\;\forall \xi\in\mathcal{U}:\text{ }|H(\xi)-\xi|<
 \gamma,\text{ }|1-(H^{-1})'(H(\xi))|<\gamma\}.\end{eqnarray*}
 One immediately verifies  that there exists a sequence 
 $\left(S^{(n)}\right)_{n\in \N}$  of functions in $ C^{2}({\mathbb S}^{1})$
  such that $\lim_{n\rightarrow\infty} d_{\mathrm{KL}}(S^{(n)},S_{\tau})=0$, where 
  the $S^{(n)}$ are viewed as interval maps.

\emph{Norms and operators}: Let $B_{0}(\mathcal{U}):=\{f:\mathcal{U} \rightarrow
\R:\Vert f\Vert_{0}<\infty \}$ be the Banach space  with
the combined norm  $\Vert \, \cdot \, \Vert_{0}$ given by 
$\Vert f\Vert_{0}:= \Vert f\Vert_{1} +\Vert f\Vert_{BV}$, where $\Vert 
\, \cdot \, \Vert_{1}$ denotes the $L^{1}$ norm and $\Vert \, \cdot 
\, \Vert_{BV}$ the bounded variation seminorm, given by 
$\sup\{\sum_{i=1}^{n}|f(\xi_{i+1})-f(\xi_i)|:0\leq x_1
<\cdots<x_n\leq1,n\in \N \}$.  Also, let 
the weak operator norm $\Vert \, \cdot \, \Vert_{W}$ be given 
 by $
\Vert \mathcal{L}\Vert_{W}:=\sup\{ \Vert \mathcal{L} (g) \Vert_{1}:g\in B_{0}(\mathcal{U}),
\text{ }\Vert g\Vert_{BV}\leq1\}$. Finally,  for an expanding map 
$S:\mathcal{U}\rightarrow\mathcal{U}$ we define the transfer 
operator $\mathcal{L}_{S}:B_{0}(\mathcal{U}) \to B_{0}(\mathcal{U})$, for
$g \in B_{0}(\mathcal{U})$ and $\xi \in \mathcal{U}$,
by
\[ \mathcal{L}_{S}g(\xi)=\sum_{S(\eta)=\xi}|S'(\eta)|^{-1} g(\eta).\]
\emph{Continuity}:
 Firstly, note that it can be shown that $\lim_{n\rightarrow\infty}||
 \mathcal{L}_{S^{(n)}}-\mathcal{L}_{S_{\tau}}||_W=0$ 
 (see comment (a) on page 143 of \cite{KL}).
Furthermore, by
\cite[Corollary 1]{KL}, we have that for each $t \in {\mathbb R}$ fixed, that 
the leading eigenvalues of the operators $\mathcal{L}_{S^{(n)}}$
converge to the leading eigenvalue of $\mathcal{L}_{S_{\tau}}$. That 
is, 
\[\lim_{n\rightarrow\infty} P(t\log |(S^{(n)})'|)=P(t\log 
|S_{\tau}'|).\]

\emph{Local uniform convergence}:
Recall that  the map given by
\[t \mapsto\beta_{S^{(n)}}(t):=  P(-t \log (S^{(n)})') / \log 2\] is  
differentiable and convex.  Using the above `Continuity', we then have that $\lim_{n\to \infty}
\beta_{S^{(n)}}(t) = 
\beta_{S_{\tau}}(t)$, for each $t 
\in \R$ fixed. Since pointwise convergence of sequences of 
differentiable 
convex functions implies local uniform convergence (see \cite[Theorem 10.8]{R}), we now 
conclude that 
\[\lim _{n\to\infty}\beta_{S^{(n)}}( 
s_{0}(S^{(n)}))+ 
s_{0}(S^{(n)}) = \beta_{S_{\tau}}(s_{0}(S_{\tau}))+ 
s_{0}(S_{\tau})
.\]
 Since $\beta_{S}(s_0(S))+ 
s_{0}(S) = \dim_{H}(\mathcal{D}_{\sim}(S,T))$, this finishes the proof of the proposition.

\end{document}